\documentclass[11pt]{article}

\title{Quantitative embedded contact homology}
\author{Michael Hutchings} \date{}

\usepackage{amssymb}
\usepackage{latexsym}
\usepackage{amsmath}
\usepackage{amsthm}
\usepackage{amscd}

\newcommand{\mc}[1]{{\mathcal #1}}

\numberwithin{equation}{section}

\newtheorem{theorem}{Theorem}[section]
\newtheorem{proposition}[theorem]{Proposition}
\newtheorem{corollary}[theorem]{Corollary}
\newtheorem{lemma}[theorem]{Lemma}
\newtheorem{lemma-definition}[theorem]{Lemma-Definition}

\newtheorem{conjecture}[theorem]{Conjecture}

\theoremstyle{definition}
\newtheorem{definition}[theorem]{Definition}
\newtheorem{remark}[theorem]{Remark}

\newtheorem{example}[theorem]{Example}

\newtheorem{update}[theorem]{Update}

\newcommand{\eqdef}{\;{:=}\;}

\renewcommand{\frak}{\mathfrak}

\newcommand{\C}{{\mathbb C}}

\newcommand{\R}{{\mathbb R}}
\newcommand{\N}{{\mathbb N}}
\newcommand{\Z}{{\mathbb Z}}
\newcommand{\op}{\operatorname}

\newcommand{\Hom}{\op{Hom}}
\newcommand{\Ker}{\op{Ker}}

\newcommand{\tensor}{\otimes}

\newcommand{\bpm}{\begin{pmatrix}}
\newcommand{\epm}{\end{pmatrix}}

\renewcommand{\epsilon}{\varepsilon}

\begin{document}

\setcounter{tocdepth}{2}

\maketitle

\begin{abstract}
  Define a ``Liouville domain'' to be a compact exact symplectic
  manifold with contact-type boundary.  We use embedded contact
  homology to assign to each four-dimensional Liouville domain (or
  subset thereof) a sequence of real numbers, which we call ``ECH
  capacities''.  The ECH capacities of a Liouville domain are defined
  in terms of the ``ECH spectrum'' of its boundary, which measures the
  amount of symplectic action needed to represent certain classes in
  embedded contact homology.  Using cobordism maps on embedded contact
  homology (defined in joint work with Taubes), we show that the ECH
  capacities are monotone with respect to symplectic embeddings.  We
  compute the ECH capacities of ellipsoids, polydisks, certain subsets
  of the cotangent bundle of $T^2$, and disjoint unions of examples
  for which the ECH capacities are known.  The resulting symplectic
  embedding obstructions are sharp in some interesting cases, for
  example for the problem of embedding an ellipsoid into a ball (as
  shown by McDuff-Schlenk) or embedding a disjoint union of balls into
  a ball.  We also state and present evidence for a conjecture under
  which the asymptotics of the ECH capacities of a Liouville domain
  recover its symplectic volume.
\end{abstract}

\section{Introduction}
\label{sec:intro}

Define a {\em Liouville domain\/} to be a compact symplectic manifold
$(X,\omega)$ such that $\omega$ is exact, and there exists a contact
form $\lambda$ on $\partial X$ with $d\lambda=\omega|_{\partial X}$.
In this paper we introduce a new obstruction to symplectically
embedding one four-dimensional Liouville domain into another, which
turns out to be sharp in some interesting cases.  For background on
symplectic embedding questions more generally we refer the reader to
\cite{chls} for an extensive discussion.

\subsection{The main theorem}

If $(X,\omega)$ is a four-dimensional Liouville domain, we use
embedded contact homology to define a sequence of real numbers
\[
0 = c_0(X,\omega) < c_1(X,\omega) \le c_2(X,\omega) \le \cdots \le \infty
\]
which we call the (distinguished) {\em ECH capacities\/} of
$(X,\omega)$.  The precise definition of these numbers is given in
\S\ref{sec:dbp}.  Our main result is:

\begin{theorem}
\label{thm:main}
Let $(X_0,\omega_0)$ and $(X_1,\omega_1)$ be four-dimensional
Liouville domains.  Suppose there is a symplectic embedding of
$(X_0,\omega_0)$ into the interior of $(X_1,\omega_1)$.  Then
\[
c_k(X_0,\omega_0) \le c_k(X_1,\omega_1)
\]
for each positive integer $k$, and the inequality is strict when
$c_k(X_0,\omega_0)<\infty$.
\end{theorem}

Note that in Theorem~\ref{thm:main}, the four-manifolds $X_0$ and
$X_1$ and their boundaries are not assumed to be connected.  The proof
of Theorem~\ref{thm:main} uses cobordism maps on embedded contact
homology induced by ``weakly exact symplectic cobordisms'', which are
defined using Seiberg-Witten theory by the construction in
\cite{cc1,cc2}.

\subsection{Examples of ECH capacities}
\label{sec:eechc}

To see what Theorem~\ref{thm:main} tells us, we now present some
computations of ECH capacities.  Given positive real numbers $a,b$,
define the ellipsoid
\begin{equation}
\label{eqn:ellipsoid}
E(a,b) \eqdef \left\{(z_1,z_2)\in\C^2 \;\bigg|\; \frac{\pi|z_1|^2}{a} +
  \frac{\pi|z_2|^2}{b}\le 1\right\}.
\end{equation}
In particular, define the ball
\[
B(a) \eqdef E(a,a).
\]
Also define the polydisk
\begin{equation}
\label{eqn:polydisk}
P(a,b)\eqdef \left\{(z_1,z_2)\in\C^2 \;\big|\; \pi|z_1|^2 \le a, \;\;
  \pi|z_2|^2 \le b\right\}.
\end{equation}
All of these examples are given the standard symplectic form
$\omega=\sum_{i=1}^2 dx_i\,dy_i$ on $\R^4=\C^2$.  The first two are
Liouville domains, because the $1$-form
\begin{equation}
\label{eqn:Liouville}
\lambda=\frac{1}{2}\sum_{i=1}^2(x_i\,dy_i-y_i\,dx_i)
\end{equation}
restricts to a contact form on the boundary of any smooth star-shaped
domain.  The polydisk is not quite a Liouville domain because its
boundary is only piecewise smooth.  However, as explained in
\S\ref{sec:mgd}, the definition of ECH capacities and
Theorem~\ref{thm:main} extend to arbitrary  subsets of
symplectic four-manifolds.  (One expects to still get decent
symplectic embedding obstructions for examples such as polydisks that
can be approximated by Liouville domains.)

To describe the ECH capacities of the ellipsoid, let $(a,b)_k$ denote
the $k^{th}$ smallest entry in the matrix of real numbers
$(am+bn)_{m,n\in\N}$.  We then have:

\begin{proposition}
\label{prop:ellipsoid}
The ECH capacities of an ellipsoid are given by
\[
c_k(E(a,b)) = (a,b)_{k+1}.
\]
\end{proposition}

Note that in the definition of ``$k^{th}$ smallest'' we count with
repetitions.  For example:

\begin{corollary}
\label{cor:ball}
The ECH capacities of a ball are given by
\[
c_k(B(a)) = da,
\]
where $d$ is the unique nonnegative integer such that
\[
\frac{d^2+d}{2} \le k \le \frac{d^2+3d}{2}.
\]
\end{corollary}

Next we have:

\begin{theorem}
\label{thm:polydisk}
The ECH capacities of a polydisk are given by
\[
c_k(P(a,b)) = \min\left\{am+bn\;\big|\; (m,n)\in\N^2, \;
  (m+1)(n+1)\ge k+1\right\}.
\]
\end{theorem}

Finally, to compute the ECH capacities of a disjoint union of examples
whose ECH capacities are known, one can use:

\begin{proposition}
\label{prop:du}
Let $(X_i,\omega_i)$ be four-dimensional Liouville domains for
$i=1,\ldots,n$.  Then
\[
c_k\left(\coprod_{i=1}^n(X_i,\omega_i)\right) = \max
\left\{\sum_{i=1}^nc_{k_i}\left(X_i,\omega_i\right) \;\bigg|\;
\sum_{i=1}^n k_i=k\right\}.
\]
\end{proposition}

\subsection{Examples of symplectic embedding obstructions}

One can now plug the above numbers into Theorem~\ref{thm:main} to get
explicit (but subtle, number-theoretic) obstructions to symplectic
embeddings.

\subsubsection{An ellipsoid into a ball (or ellipsoid)}
\label{sec:eb}

For example, consider the problem of symplectically embedding an
ellipsoid into a ball.  By scaling, we can encode this problem into a single
function as follows: Given $a>0$, define $f(a)$ to be the infimum over
$c\in\R$ such that the ellipsoid $E(a,1)$ symplectically embeds into
the ball $B(c)$.  By Theorem~\ref{thm:main},
Proposition~\ref{prop:ellipsoid}, and Corollary~\ref{cor:ball}, we
have
\begin{equation}
\label{eqn:eb}
f(a)  \ge \sup_{k=2,3,\ldots}\frac{(a,1)_k}{(1,1)_k} =
\sup_{d=1,2,\ldots}\frac{1}{d}(a,1)_{(d^2+3d+2)/2}.
\end{equation}
On the other hand, McDuff-Schlenk \cite{ms} computed the function $f$
explicitly, obtaining a beautiful and complicated answer involving
Fibonacci numbers.  Using their result, they confirmed that the
reverse inequality in \eqref{eqn:eb} holds.  Thus the ECH capacities
give a sharp embedding obstruction in this case.

\begin{update}
  More recently, McDuff \cite{mh} has shown that
  the ECH obstruction to symplectically embedding one ellipsoid into
  another is sharp: $\op{int}(E(a,b))$ symplectically embeds into
  $E(c,d)$ if and only if $(a,b)_k\le (c,d)_k$ for all $k$.
\end{update}

\subsubsection{A polydisk into a ball}
\label{sec:pb}

Next let us consider the problem of symplectically embedding a
polydisk into a ball.  Given $a>0$, define $g(a)$ to be the infimum
over $c\in\R$ such that the polydisk $P(a,1)$ symplectically embeds
into the ball $B(c)$.  By Theorems~\ref{thm:main} and
\ref{thm:polydisk} and Corollary~\ref{cor:ball}, we have
\begin{equation}
\label{eqn:po}
g(a) \ge \sup_{d=1,2,\ldots}
\min\left\{\frac{am+n}{d} \;\bigg|\; (m,n)\in\N^2,\;\;
  (m+1)(n+1)\ge \frac{(d+1)(d+2)}{2}\right\}.
\end{equation}
Simple calculations in \S\ref{sec:simplecalculations} then deduce:

\begin{proposition}
\label{prop:pb}
The obstruction to symplectically embedding a polydisk into a ball
satisfies
\begin{equation}
\label{eqn:ga}
g(a) \ge \left\{\begin{array}{cl}
2, & 1\le a\le 2,\\
1+\frac{a}{2}, & 2\le a \le 3,\\
\frac{3}{2} + \frac{a}{3}, & 3\le a \le 4.
\end{array}\right.
\end{equation}
\end{proposition}

Note that when $a\neq 2$ this is better than the lower bound $g(a)\ge
\sqrt{2a}$ obtained by considering volumes.  For $a$ slightly larger
than $4$, a more complicated calculation which we omit shows that the
best bound that can be obtained from \eqref{eqn:po} is
\[
g(a) \ge \frac{19}{12} + \frac{5a}{16},
\]
which comes from taking $d=48$ in \eqref{eqn:po}.  We do not know much
about the right hand side of \eqref{eqn:po} for larger $a$, although
we do know that it is always at least $\sqrt{2a}$, see
\S\ref{sec:volumeintro} below.  By analogy with \cite{ms} one might
guess that $g(a)=\sqrt{2a}$ when $a$ is sufficiently large.

\begin{remark}
  We do not know to what extent the bound \eqref{eqn:po} is sharp.  In
  general, the obstruction from Theorem~\ref{thm:main} to
  embedding a polydisk into an ellipsoid is not always sharp.  For
  example, Proposition~\ref{prop:ellipsoid} and
  Theorem~\ref{thm:polydisk} imply that $P(1,1)$ and $E(1,2)$ have the
  same ECH capacities, namely
\[
0,1,2,2,3,3,4,4,4,5,5,5,\ldots.
\]
  Thus the ECH capacities
  give no obstruction to symplectically embedding $P(1,1)$ into
  $E(a,2a)$ for any $a>1$, and in particular tell us nothing more than
  volume comparison.  However the Ekeland-Hofer capacities give an
  obstruction to symplectically embedding $P(1,1)$ into $E(a,2a)$
  whenever $a<3/2$.  (The Ekeland-Hofer capacities of $P(1,1)$ are
  $1,2,3,\ldots$, while those of $E(a,2a)$ are
  $a,2a,2a,3a,4a,4a,\ldots$, see \cite{chls}.)  Note that $P(1,1)$
  does symplectically embed into $E(a,2a)$ whenever $a\ge 3/2$.  Indeed,
  with the conventions of \eqref{eqn:ellipsoid} and
  \eqref{eqn:polydisk}, $P(1,1)$ is a subset of $E(3/2,3)$.
\end{remark}

\subsubsection{A disjoint union of balls into a ball}

The ECH capacities give the following obstruction to
symplectically embedding a disjoint union of balls into a ball:

\begin{proposition}
\label{prop:packing}
Suppose there is a symplectic embedding of $\coprod_{i=1}^n B(a_i)$
into the interior of $B(1)$.  Then
\begin{equation}
\label{eqn:packing}
\sum_{i=1}^nd_ia_i < d
\end{equation}
whenever $(d_1,\ldots,d_n,d)$ are nonnegative integers (not all zero)
satisfying
\[
\sum_{i=1}^n(d_i^2+d_i) \le d^2+3d.
\]
\end{proposition}

\begin{proof}
  Let $k_i\eqdef (d_i^2+d_i)/2$ for $i=1,\ldots,n$, let $k\eqdef
  \sum_{i=1}^nk_i$, and let $k'\eqdef (d^2+3d)/2$.  By
  Corollary~\ref{cor:ball} we have $c_{k_i}(B(a_i))=d_ia_i$ and
  $c_{k'}(B(1))=d$.  Then
\[
\sum_{i=1}^nd_ia_i = 
\sum_{i=1}^n c_{k_i}(B(a_i))
\le c_k\left(\coprod_{i=1}^nB(a_i)\right) <
c_k(B(1)) \le c_{k'}(B(1))=d.
\]
Here the first inequality holds by Proposition~\ref{prop:du}, the
second inequality by Theorem~\ref{thm:main}, and the third inequality
by our assumption that $k\le k'$.
\end{proof}

\begin{remark}
  Proposition~\ref{prop:packing} is not new and, as explained to me by
  Dusa McDuff, can also be deduced by applying Taubes's
  ``Seiberg-Witten = Gromov'' theorem \cite{swgr} to a symplectic
  blowup of ${\mathbb CP}^2$.  The interesting point is that
  Proposition~\ref{prop:packing}, and thus ECH, gives a sharp
  obstruction.  Indeed, it follows from work of Biran \cite[Thm.\
  3.2]{biran} that there exists a symplectic embedding of
  $\coprod_{i=1}^nB(a_i)$ into $B(1+\epsilon)$ for all $\epsilon>0$
  if:
\begin{description}
\item{(i)} $\sum_{i=1}^na_i^2\le 1$, i.e.\ the volume of
  $\coprod_iB(a_i)$ is less than or equal to that of $B(1)$, and
\item{(ii)} the inequality $\sum_{i=1}^nd_ia_i\le d$ holds for all
  tuples of nonnegative integers $(d_1,\ldots,d_n,d)$ satisfying
  $\sum_{i=1}^nd_i=3d-1$ and $\sum_{i=1}^nd_i^2=d^2+1$.
\end{description}
(As explained in \cite[\S1.2]{ms}, results of \cite{ll,mp} imply that
one can replace the inequalities (ii) above by a certain subset
thereof.)  But Proposition~\ref{prop:packing} implies that conditions
(i) and (ii) are also necessary for the existence of a symplectic
embedding.  Note here that by Proposition~\ref{prop:somewhere} below,
the inequalities \eqref{eqn:packing} imply the volume constraint (i).
\end{remark}

\subsection{More examples of ECH capacities}
\label{sec:dualintro}

We can also compute the ECH capacities of certain subsets of the
cotangent bundle of $T^2=\R^2/\Z^2$, such as the unit disk bundle,
using results from \cite{t3}.  Let $\|\cdot\|$ be a norm on $\R^2$,
regarded as a translation-invariant norm on $TT^2$.  Let $\|\cdot\|^*$
denote the dual norm on $(\R^2)^*$, which we regard as a
translation-invariant norm on $T^*T^2$.  That is, if $\zeta\in
T_q^*T^2$, then
\[
\|\zeta\|^* = \max\left\{\langle\zeta,v\rangle \;\big|\; v\in T_qT^2,
  \; \|v\|\le 1\right\}.
\]
Define
\[
T_{\|\cdot\|^*} \eqdef \left\{\zeta\in T^*T^2 \;\big|\; \|\zeta\|^*\le
    1\right\},
\]
with symplectic form obtained by restricting the standard symplectic
form $\omega=\sum_{i=1}^2dp_i\,dq_i$ on $T^*T^2$.  Here $q_1,q_2$
denote the standard coordinates on $T^2$, and $p_1,p_2$ denote the
corresponding coordinates on the cotangent fibers.

If $\|\cdot\|$ is smooth, then the unit ball in the dual norm
$\|\cdot\|^*$ on $\R^2$ is smooth, and $T_{\|\cdot\|^*}$ is a
Liouville domain, because $\lambda=\sum_{i=1}^2p_idq_i$ restricts to a
contact form on the boundary.  For example, if $\|\cdot\|$ is the
Euclidean norm, then $T_{\|\cdot\|^*}$ is the unit disk bundle in the
cotangent bundle of $T^2$ with the standard flat metric.

\begin{theorem}
\label{thm:dual}
If $\|\cdot\|$ is a norm on $\R^2$, then
\begin{equation}
\label{eqn:dual}
c_k\left(T_{\|\cdot\|^*}\right) =
\min\left\{\ell_{\|\cdot\|}(\Lambda) \,\big|\, |P_\Lambda\cap\Z^2|=k+1
\right\}.
\end{equation}
Here the minimum is over convex polygons $\Lambda$ in $\R^2$ with
vertices in $\Z^2$, and $P_\Lambda$ denotes the closed region bounded
by $\Lambda$.  Also $\ell_{\|\cdot\|}(\Lambda)$ denotes the length of
$\Lambda$ in the norm $\|\cdot\|$.
\end{theorem}

It is an interesting problem to understand the ECH capacities of the
unit disk bundle in the cotangent bundle of more general surfaces than
flat $T^2$.

\subsection{Volume conjecture}
\label{sec:volumeintro}

In all of the examples considered above, it turns out that the
asymptotic behavior of the symplectic embedding obstruction given by
Theorem~\ref{thm:main} as $k\to\infty$ simply recovers the necessary
condition that the volume of $(X_0,\omega_0)$ be less than or equal to
that of $(X_1,\omega_1)$.  Here the volume of a four-dimensional Liouville domain
$(X,\omega)$ is defined by
\[
\op{vol}(X,\omega) = \frac{1}{2}\int_X \omega\wedge\omega.
\]
The conjectural more general phenomenon is that the asymptotics of the
ECH capacities are related to volume as follows:

\begin{conjecture}
\label{conj:Liouville}
Let $(X,\omega)$ be a four-dimensional Liouville domain such that
$c_k(X,\omega)<\infty$ for all $k$.  Then
\[
\lim_{k\to\infty}\frac{c_k(X,\omega)^2}{k} = 4\op{vol}(X,\omega).
\]
\end{conjecture}

It is not hard to check this for an ellipsoid, cf.\
Remark~\ref{rem:ellipsoidVolume}.  It is also easy to check this for a
polydisk (even though the conjecture is not applicable here since a
polydisk is not quite a Liouville domain).  In \S\ref{sec:volume} we
further confirm that this conjecture holds for the examples in
Theorem~\ref{thm:dual}, as well as for any disjoint union or subset of
examples for which the conjecture holds.  Note that the hypothesis
that $c_k(X,\omega)<\infty$ for all $k$ holds only if the first Chern
class (not the ECH capacity) $c_1(X,\omega)\in H^2(X;\Z)$ restricts to
a torsion class in $H^2(\partial X;\Z)$, see Remark~\ref{rem:torsion}.

Conjecture~\ref{conj:Liouville} is related to the question of whether
the Weinstein conjecture in three dimensions \cite{taubes:wc} can be
refined to show that a closed contact 3-manifold has a
Reeb orbit with an explicit upper bound on the length, see
Remark~\ref{rem:qw}.

\subsection{Contents of the paper}

There are in fact two basic ways to define ECH capacities of a
four-dimensional Liouville domain $(X,\omega)$: in addition to the
``distinguished'' ECH capacities $c_k(X,\omega)$ discussed above,
there is also a more rudimentary notion which we call the ``full ECH
capacities'' and which we denote by $\widetilde{c}_k(X,\omega)$.  The
full ECH capacities satisfy an analogue of Theorem~\ref{thm:main}, but
only under the additional assumption that if $\varphi$ denotes the
symplectic embedding in question, then
$X_1\setminus\varphi(\op{int}(X_0))$ is diffeomorphic to a product
$[0,1]\times Y^3$.  The numbers $c_k(X,\omega)$ are a certain
carefully selected subset of the numbers $\widetilde{c}_k(X,\omega)$
for which the more general statement of Theorem~\ref{thm:main} is
true.

Both the full and distinguished ECH capacities of a four-dimensional
Liouville domain $(X,\omega)$ with boundary $Y$ are defined in terms
of the embedded contact homology of $(Y,\lambda)$, where $\lambda$ is
a contact form on $Y$ with $d\lambda=\omega|_Y$.  In \S\ref{sec:ECH}
we recall the necessary material about embedded contact homology.

In \S\ref{sec:full} we associate to a closed contact 3-manifold
$(Y,\lambda)$ a sequence of numbers $\widetilde{c}_k(Y,\lambda)$,
which we call its ``full ECH spectrum''; these numbers measure the
amount of symplectic action needed to represent certain classes in the
embedded contact homology of $(Y,\lambda)$.  The full ECH capacities
of a four-dimensional Liouville domain are then defined to be the full
ECH spectrum of its boundary.  Proposition~\ref{prop:ellipsoid} above
regarding the ECH capacities of ellipsoids is equivalent to
Proposition~\ref{prop:ellipsoidCapacity} which is proved in this
section.

In \S\ref{sec:distinguished} we give the crucial definition of the
``distinguished ECH spectrum'' of a closed contact 3-manifold
$(Y,\lambda)$ with nonvanishing ECH contact invariant (e.g.\ the
boundary of a Liouville domain).  The distinguished ECH capacities of
a four-dimensional Liouville domain are then defined to be the
distinguished ECH spectrum of its boundary.  This section also gives
the proof of Theorem~\ref{thm:main}; once the correct definitions are
in place, this is a simple application of the machinery of ECH
cobordism maps from \cite{cc2}.  Finally, this section explains how to
extend the definition of (distinguished) ECH capacities and
Theorem~\ref{thm:main} to arbitrary subsets of symplectic
four-manifolds.

In \S\ref{sec:interlude} we compute the (distinguished) ECH spectrum
of a disjoint union of contact 3-manifolds, which implies
Proposition~\ref{prop:du} above on the ECH capacities of a disjoint
union of Liouville domains.  In \S\ref{sec:t3} we prove
Theorem~\ref{thm:dual} regarding the ECH capacities of certain subsets
of $T^*T^2$.  In \S\ref{sec:pd} we prove Theorem~\ref{thm:polydisk} on
the ECH capacities of a polydisk.  Proposition~\ref{prop:pb} above on
the obstruction to symplectically embedding a polydisk into a ball is
also proved in this section.  Finally, in \S\ref{sec:volume} we
discuss the volume conjecture \ref{conj:Liouville} and several
variants, and present some evidence for them.

\paragraph{Acknowledgments.}
I thank Dusa McDuff for introducing me to this problem and for helpful
discussions, Felix Schlenk and Cliff Taubes for additional helpful
discussions, and MSRI for its hospitality.  This work was partially
supported by NSF grant DMS-0806037.

\section{ECH preliminaries}
\label{sec:ECH}

We now review the necessary background on embedded contact homology.

\subsection{Definition of ECH}

Let $Y$ be a closed oriented 3-manifold.  A {\em contact form\/} on
$Y$ is a $1$-form $\lambda$ on $Y$ with $\lambda\wedge d\lambda>0$
everywhere.  This determines a {\em contact structure\/}, namely the
oriented $2$-plane field $\xi=\Ker(\lambda)$.  We call the pair
$(Y,\lambda)$ a ``contact 3-manifold'', although it is perhaps more
usual to refer to the pair $(Y,\xi)$ this way.

The contact form $\lambda$ determines the {\em Reeb vector field\/}
$R$ characterized by $d\lambda(R,\cdot)=0$ and $\lambda(R)=1$.  A {\em
  Reeb orbit\/} is a closed orbit of the Reeb vector field $R$, i.e.\
a map $\gamma:\R/T\Z\to Y$ for some $T>0$ with
$\gamma'(t)=R(\gamma(t))$, modulo reparametrization.  A Reeb orbit is
{\em nondegenerate\/} if its linearized return map, regarded as an
endomorphism of the $2$-dimensional symplectic vector space
$(\xi_{\gamma(0)},d\lambda)$, does not have $1$ as an eigenvalue.  A
nondegenerate Reeb orbit is called {\em hyperbolic\/} if its
linearized return map has real eigenvalues; otherwise it is called
{\em elliptic\/}.  We say that the contact form $\lambda$ is
nondegenerate if all Reeb orbits are nondegenerate.

If $Y$ is a closed oriented $3$-manifold with a nondegenerate contact
form $\lambda$, and if $\Gamma\in H_1(Y)$, then the {\em embedded
  contact homology\/} with $\Z/2$-coefficients, which we denote by
$ECH(Y,\lambda,\Gamma)$, is defined.  (ECH can also be defined over
$\Z$, see \cite[\S9]{obg2}, but $\Z/2$ coefficients are sufficient for
the applications in this paper.)  This is the homology of a chain
complex which is generated over $\Z/2$ by finite sets of pairs
$\alpha=\{(\alpha_i,m_i)\}$ where the $\alpha_i$'s are distinct
embedded Reeb orbits, the $m_i$'s are positive integers, $m_i=1$
whenever $\alpha_i$ is hyperbolic, and
\[
\sum_im_i[\alpha_i]=\Gamma\in H_1(Y).
\]
We call such an $\alpha$ an {\em ECH generator\/}.
We often use the multiplicative
notation $\alpha=\prod_i\alpha_i^{m_i}$, even though the grading and
differential on the chain complex do not behave simply with respect to
this sort of multiplication.

To define the chain complex differential $\partial$ one chooses a
generic almost complex structure $J$ on $\R\times Y$ which is
``admissible'', meaning that $J$ is $\R$-invariant, $J(\partial_s)=R$
where $s$ denotes the $\R$ coordinate, and $J$ sends $\xi$ to itself,
rotating positively with respect to the orientation $d\lambda$ on
$\xi$.  The coefficient $\langle\partial\alpha,\beta\rangle$ of the
differential is then a count of $J$-holomorphic curves in $\R\times Y$
which have ECH index $1$ and which as currents are asymptotic to
$\R\times\alpha$ as $s\to\infty$ and asymptotic to $\R\times\beta$ as
$s\to-\infty$.  The detailed definition of the differential is given
for example in \cite[\S7]{obg1}, using the ECH index defined in
\cite{pfh2,ir}.  We denote this chain complex by
$ECC(Y,\lambda,\Gamma,J)$, and its homology by
$ECH(Y,\lambda,\Gamma)$.

The $\Z/2$-module $ECH(Y,\lambda,\Gamma)$ has a relative
$\Z/d$-grading, where $d$ denotes the divisibility of
$c_1(\xi)+2\op{PD}(\Gamma)$ in $H^2(Y;\Z)/\op{Torsion}$.  The detailed
definition of the grading will not be needed here and can be found in
\cite{pfh2,ir}.

Although the differential on the chain complex
$ECC(Y,\lambda,\Gamma,J)$ depends on $J$, the homology
$ECH(Y,\lambda,\Gamma)$ does not.  This follows from a much stronger
theorem of Taubes
\cite{taubes:echswf1,taubes:echswf2,taubes:echswf3,taubes:echswf4}
asserting that there is a canonical isomorphism between embedded
contact homology and a version of Seiberg-Witten Floer cohomology as
defined by Kronheimer-Mrowka \cite{km}. Namely, if $Y$ is connected
then there is a canonical isomorphism of relatively graded
$\Z/2$-modules
\begin{equation}
\label{eqn:echswf}
ECH_*(Y,\lambda,\Gamma) \stackrel{\simeq}{\longrightarrow}
\widehat{HM}^{-*}(Y,\frak{s}_\xi+\op{PD}(\Gamma)),
\end{equation}
where the right hand side denotes Seiberg-Witten Floer cohomology with
$\Z/2$-coefficients, and $\frak{s}_\xi$ is a spin-c structure
determined by the contact structure.  (This is also true with $\Z$
coefficients.)  As shown in \cite{cc2}, it follows from Taubes's proof
of \eqref{eqn:echswf} and the invariance properties of
$\widehat{HM}$ that the versions of $ECH(Y,\lambda,\Gamma)$ defined
using different almost complex structures $J$ are canonically
isomorphic to each other. 

In this paper we are almost exclusively concerned with the case $\Gamma=0$.

\subsection{Some additional structure on ECH}

There is a canonical element
\[
c(\xi) \eqdef [\emptyset]\in ECH(Y,\lambda,0),
\]
called the {\em ECH contact invariant\/}, represented by the ECH
generator consisting of the empty set of Reeb orbits.  This is a cycle
in the ECH chain complex because any holomorphic curve counted by the
differential must have at least one positive end, c.f.\
\S\ref{sec:filtered} below.  The homology class $[\emptyset]$ depends
only on the contact structure $\xi$ (although not just on $Y$), and
agrees with an analogous contact invariant in Seiberg-Witten Floer
cohomology \cite{echswf5}.

If $Y$ is connected, then there is a degree $-2$ map
\begin{equation}
\label{eqn:U}
U:ECH(Y,\lambda,\Gamma) \longrightarrow ECH(Y,\lambda,\Gamma).
\end{equation}
This is induced by a chain map which is defined similarly to the
differential, but instead of counting holomorphic curves in $\R\times
Y$ with ECH index one modulo translation, it counts holomorphic curves
in $\R\times Y$ with ECH index two that pass through a chosen generic
point $z\in\R\times Y$, see \cite[\S2.5]{wh}.  Under the isomorphism
\eqref{eqn:echswf}, the $U$ map \eqref{eqn:U} agrees with an analogous
map on Seiberg-Witten Floer cohomology \cite{echswf5}.

If $(Y,\lambda)$ has connected components $(Y_i,\lambda_i)$ for
$i=1,\ldots,n$, then there are $n$ different $U$ maps
$U_1,\ldots,U_n$, where $U_i$ is defined by taking $z\in \R\times Y_i$.  The
different maps $U_i$ commute.  Note also that in this case one has a
canonical isomorphism of chain complexes
\begin{equation}
\label{eqn:tensorchain}
ECC(Y_1,\lambda_1,\Gamma_1,J_1)\tensor\cdots\tensor
ECC(Y_n,\lambda_n,\Gamma_n,J_n)
\stackrel{\simeq}{\longrightarrow} ECC(Y,\lambda,\Gamma,J),
\end{equation}
which sends a tensor product of ECH generators on the left hand side
to their union on the right, where $\Gamma=\sum_{i=1}^n\Gamma_i$ and
$J$ restricts to $J_i$ on $\R\times Y_i$.  Since we are working with
field coefficients, this gives a canonical isomorphism on homology
\begin{equation}
\label{eqn:tensor}
ECH(Y,\lambda,\Gamma) = ECH(Y_1,\lambda_1,\Gamma_1) \tensor \cdots \tensor
ECH(Y_n,\lambda_n,\Gamma_n).
\end{equation}
Under this identification, $U_i$ is the tensor product of the $U$ map
for $(Y_i,\lambda_i)$ with the identity maps on the other factors.

\subsection{Filtered ECH}
\label{sec:filtered}

If $\alpha=\{(\alpha_i,m_i)\}$ is a generator of the
ECH chain complex, its {\em symplectic action\/} is defined by
\[
\mc{A}(\alpha) \eqdef \sum_im_i\int_{\alpha_i}\lambda.
\]
The ECH differential (for any generic admissible $J$) decreases the
action, i.e. if $\langle\partial\alpha,\beta\rangle\neq 0$ then
$\mc{A}(\alpha)\ge\mc{A}(\beta)$.  This is because if $C$ is a
$J$-holomorphic curve counted by $\langle\partial\alpha,\beta\rangle$,
then $d\lambda|_C\ge 0$ everywhere.  (In fact if
$\langle\partial\alpha,\beta\rangle\neq 0$ then the strict inequality
$\mc{A}(\alpha)>\mc{A}(\beta)$ holds, because $d\lambda$ vanishes
identically on $C$ if and only if the image of $C$ is $\R$-invariant,
in which case $C$ has ECH index zero and so does not contribute to the
differential.)  Thus for any real number $L$, it makes sense to define
the {\em filtered ECH\/}
\[
ECH^{L}(Y,\lambda,\Gamma)
\]
to be the homology of the subcomplex $ECC^L(Y,\lambda,\Gamma,J)$ of
the ECH chain complex spanned by generators with action (strictly)
less than $L$.  It is shown in \cite{cc2} that
$ECH^{L}(Y,\lambda,\Gamma)$ does not depend on the choice of generic
admissible $J$ (although unlike the usual ECH it can change when one
deforms the contact form $\lambda$).  For $L<L'$ the inclusion of
chain complexes (for a given $J$) induces a map
\[
\imath_*:ECH^{L}(Y,\lambda,\Gamma) \longrightarrow
ECH^{L'}(Y,\lambda,\Gamma).
\]
It is shown in \cite{cc2} that this map does not depend on the choice
of $J$.  The usual ECH is recovered as the direct limit
\[
ECH(Y,\lambda,\Gamma) = \lim_{\to}ECH^{L}(Y,\lambda,\Gamma).
\]
Also, if $c$ is a positive constant, then there is a canonical
``scaling'' isomorphism
\begin{equation}
\label{eqn:scaling}
s:ECH^{L}(Y,\lambda,\Gamma) \stackrel{\simeq}{\longrightarrow}
 ECH^{cL}(Y,c\lambda,\Gamma).
\end{equation}
The reason is that an admissible almost complex structure $J$ for
$\lambda$ determines an admissible almost complex structure for
$c\lambda$, such that the obvious identification of Reeb orbits gives
an isomorphism of chain complexes.  Again, it is shown in \cite{cc2}
that the resulting map \eqref{eqn:scaling} does not depend on the
choice of $J$.

\subsection{Weakly exact symplectic cobordisms}
\label{sec:esc}

Let $(Y_+,\lambda_+)$ and $(Y_-,\lambda_-)$ be closed contact
3-manifolds.

\begin{definition}
\label{def:exact}
An {\em exact symplectic cobordism\/} from $(Y_+,\lambda_+)$ to
$(Y_-,\lambda_-)$ is a compact symplectic $4$-manifold $(X,\omega)$
with $\partial X = Y_+-Y_-$, such that there exists a $1$-form
$\lambda$ on $X$ with $d\lambda=\omega$ and
$\lambda|_{Y_\pm}=\lambda_\pm$.
\end{definition}

It is shown in \cite{cc2} that if the contact forms $\lambda_\pm$ are
nondegenerate, then an exact symplectic cobordism as above induces
maps of ungraded $\Z/2$-modules
\begin{equation}
\label{eqn:cc2}
\bigoplus_{\Gamma_+\in H_1(Y_+)}ECH^L(Y_+,\lambda_+,\Gamma_+) \longrightarrow
\bigoplus_{\Gamma_-\in H_1(Y_-)}ECH^L(Y_-,\lambda_-,\Gamma_-)
\end{equation}
satisfying various axioms.  The idea of the construction is as
follows.  Consider the ``symplectization completion'' of $X$ defined
by
\begin{equation}
\label{eqn:completion}
\overline{X} \eqdef ((-\infty,0]\times Y_-)\cup_{Y_-} X \cup_{Y_+}
([0,\infty)\times Y_+).
\end{equation}
As reviewed after Definition~\ref{def:wesc} below, the symplectic form
$\omega$ on $X$ naturally extends over $\overline{X}$ as
$d(e^s\lambda_-)$ on $(-\infty,0]\times Y_-$, where $s$ denotes the
$(-\infty,0]$ coordinate, and as $d(e^s\lambda_+)$ on
$[0,\infty)\times Y_+$.  A suitable almost complex structure $J$ on
$\overline{X}$ determines, via $\omega$, a metric on $\overline{X}$.
One then modifies $\omega$ and the metric on the ends to obtain a 2-form
$\hat{\omega}$ and a metric which are $\R$-invariant on the ends.  The map
\eqref{eqn:cc2} is now induced by a chain map which is defined by
counting solutions to the Seiberg-Witten equations on $\overline{X}$
perturbed using a large multiple of the $2$-form $\hat{\omega}$. In
the limit as the perturbation gets large, the relevant Seiberg-Witten
solutions give rise to (possibly broken) $J$-holomorphic curves in
$\overline{X}$.  The restriction of $\omega$ to any such
$J$-holomorphic curve is pointwise nonnegative.  The key fact needed
to get a well-defined map on filtered ECH is then that if $\alpha_\pm$
are smooth $1$-chains in $Y_\pm$, and if $Z$ is a smooth $2$-chain in
$X$ with $\partial Z = \alpha_+ - \alpha_-$, then
\begin{equation}
\label{eqn:keyfact}
\int_Z\omega = \int_{\alpha_+}\lambda_+ - \int_{\alpha_-}\lambda_-.
\end{equation}
Of course this holds by the exactness assumption and Stokes's theorem.

We now show that the $\Gamma_\pm=0$ component of the map
\eqref{eqn:cc2} can still be defined under a slightly weaker
assumption, in which we take $d$ of the last equation in
Definition~\ref{def:exact}:

\begin{definition}
\label{def:wesc}
A {\em weakly exact symplectic cobordism\/}
from $(Y_+,\lambda_+)$ to $(Y_-,\lambda_-)$ is a compact symplectic
$4$-manifold $(X,\omega)$ with $\partial X = Y_+ - Y_-$, such
that $\omega$ is exact and $\omega|_{Y_\pm} = d\lambda_\pm$.
\end{definition}

For example, a four-dimensional Liouville domain as we have defined it
is a weakly exact symplectic cobordism from a contact 3-manifold to
the empty set. Note that for any weakly exact symplectic cobordism $X$
as above, by a standard lemma there is an identification of a
neighborhood of $Y_+$ in $X$ with $(-\varepsilon,0]\times Y_+$ such
that on this neighborhood we have $\omega=d(e^s\lambda_+)$, where $s$
denotes the $(-\varepsilon,0]$ coordinate.  Likewise a neighborhood of
$Y_-$ in $X$ can be identified with $[0,\varepsilon)\times Y_-$ so
that on this neighborhood $\lambda=d(e^s\lambda_-)$.  Thus one can
still define the symplectization completion $\overline{X}$ as in
\eqref{eqn:completion}.

\begin{theorem}
\label{thm:cob}
Let $(X,\omega)$ be a weakly exact symplectic cobordism from
$(Y_+,\lambda_+)$ to $(Y_-,\lambda_-)$, where $Y_+$ and $Y_-$ are
closed and the contact forms $\lambda_\pm$ are nondegenerate.  Then
there exist maps
\begin{equation}
\label{eqn:philgo}
\Phi^L(X,\omega): ECH^L(Y_+,\lambda_+,0) \longrightarrow
ECH^L(Y_-,\lambda_-,0)
\end{equation}
of ungraded $\Z/2$-modules, for each $L\in\R$, with the following properties:
\begin{description}
\item{(a)}
If $L<L'$ then the following diagram commutes:
\[
\begin{CD}
ECH^{L}(Y_+,\lambda_+,0) @>{\Phi^L(X,\omega)}>>
ECH^{L}(Y_-,\lambda_-,0) \\
@V{\imath_*}VV @VV{\imath_*}V \\
ECH^{L'}(Y_+,\lambda_+,0) @>{\Phi^{L'}(X,\omega)}>>
ECH^{L'}(Y_-,\lambda_-,0).
\end{CD}
\]
\end{description}
In particular, it makes sense to define the direct limit
\begin{equation}
\label{eqn:Phi}
\Phi(X,\omega) \eqdef \lim_{\to}\Phi^L(X,\omega):
ECH(Y_+,\lambda_+,0) \longrightarrow ECH(Y_-,\lambda_-,0)
\end{equation}
\begin{description}
\item{(b)}
$\Phi(X,\omega)[\emptyset]=[\emptyset]$.
\item{(c)} If $X$ is diffeomorphic to a product $[0,1]\times Y$, then
  $\Phi(X,\omega)$ is an isomorphism.
\item{(d)}
The diagram
\[
\begin{CD}
ECH(Y_+,\lambda_+,0) @>{\Phi(X,\omega)}>> ECH(Y_-,\lambda_-,0)\\
@VV{U_+}V @VV{U_-}V\\
ECH(Y_+,\lambda_+,0) @>{\Phi(X,\omega)}>> ECH(Y_-,\lambda_-,0)
\end{CD}
\]
commutes, where $U_\pm$ is the $U$ map for any of the connected
components of $Y_\pm$, as long as $U_+$ and $U_-$ correspond to the
same component of $X$.
\end{description}
\end{theorem}

\begin{proof}
  Suppose first that $Y_+$ and $Y_-$ are connected and that
  $(X,\omega)$ is exact as in Definition~\ref{def:exact}.  In this
  case we define $\Phi^L(X,\omega)$ from the map \eqref{eqn:cc2} by
  restricting to the $\Gamma_+=0$ component and projecting to the
  $\Gamma_-=0$ component.  It follows from the main theorem in
  \cite{cc2} that $\Phi^L(X,\omega)$ satisfies properties (a) and (b),
  and $\Phi(X,\omega)$ agrees with the $\Gamma_\pm=0$ component of
  the induced map on Seiberg-Witten Floer cohomology via the
  isomorphisms \eqref{eqn:echswf} on both sides.  Items (c) and (d)
  then follow from analogous results in Seiberg-Witten Floer theory
  \cite{km}.

  If $(X,\omega)$ is only weakly exact, then one can no longer define
  a map \eqref{eqn:cc2}, but one can still define a map on
  $\Gamma_\pm=0$ components as in \eqref{eqn:philgo}, again by
  perturbing the Seiberg-Witten equations on the symplectization completion
  $\overline{X}$ using a large multiple of $\hat{\omega}$.  One just needs to
  check that \eqref{eqn:keyfact} holds when $\alpha_\pm$ is
  nullhomologous in $Y_\pm$.  To do so, let $\lambda$ be a $1$-form on
  $X$ with $d\lambda=\omega$.  Then by Stokes's theorem we have
  $\int_Z\omega = \int_{\alpha_+}\lambda - \int_{\alpha_-}\lambda$.
  On the other hand, since $\lambda|_{Y_\pm}-\lambda_\pm$ is a closed $1$-form
  on $Y_\pm$ and $\alpha_\pm$ is nullhomologous in $Y_\pm$, by
  Stokes's theorem again we have
  $\int_{\alpha_\pm}(\lambda-\lambda_\pm)=0$.  Properties (a)--(d)
  hold as before.

  When $Y_+$ and $Y_-$ are not required to be connected, one can still
  construct the maps $\Phi^L(X,\omega)$ and prove properties (a) and
  (b) by deforming the Seiberg-Witten equations on $\overline{X}$
  using a large multiple of $\hat{\omega}$ as above (and we already
  know property (c) in this case).  One can then prove property (d) by
  using the interpretation of the Seiberg-Witten $U$ map in
  \cite{echswf5} (which counts index $2$ Seiberg-Witten solutions in
  $\R\times Y$ satisfying a codimension $2$ constraint at a chosen
  point) to construct a chain homotopy between the chain maps defining
  $U_+\circ\Phi(X,\omega)$ and $\Phi(X,\omega)\circ U_-$ (by counting
  index $1$ Seiberg-Witten solutions in the completed cobordism
  satisfying a codimension $2$ constraint at any point along a
  suitable path).
\end{proof}

\section{Full ECH spectrum and capacities}
\label{sec:full}

We now introduce the full ECH spectrum and capacities, as a
warmup for the distinguished ECH spectrum and capacities to be defined
in \S\ref{sec:distinguished}.

\subsection{The full ECH spectrum}
\label{sec:ECHspectrum}

Let $Y$ be a closed oriented 3-manifold with a nondegenerate contact
form $\lambda$. 

\begin{definition}
\label{def:fullspectrum}
  For each positive integer $k$, define $\widetilde{c}_k(Y,\lambda)$
  to be the infimum over all $L\in\R$ such that the image of
  $ECH^L(Y,\lambda,0)$ in $ECH(Y,\lambda,0)$ has dimension at least $k$.
  The sequence $\{\widetilde{c}_k(Y,\lambda)\}_{k=1,2,\ldots}$ is
  called the {\em full ECH spectrum\/} of $(Y,\lambda)$.
\end{definition}

\begin{remark}
\label{rem:spec}
\begin{description}
\item{(a)}
It follows from the definition that
\[
0 \le \widetilde{c}_1(Y,\lambda)\le \widetilde{c}_2(Y,\lambda) \le
\cdots \le \infty.
\]
Note that if $Y$ is connected, then
$\widetilde{c}_k(Y,\lambda)<\infty$ for all $k$ if and only if
$c_1(\xi)\in H^2(Y;\Z)$ is torsion. This is because Taubes's
isomorphism \eqref{eqn:echswf}, together with results of
Kronheimer-Mrowka \cite{km}, imply that if $Y$ is connected, then
$ECH(Y,\lambda,\Gamma)$ is infinitely generated if and only if
$c_1(\xi)+2\op{PD}(\Gamma)\in H^2(Y;\Z)$ is torsion.
\item{(b)}
It follows immediately from the
definition that
\[
\widetilde{c}_1(Y,\lambda) >0 \Longleftrightarrow c(\xi)=0\in
ECH(Y,\lambda,0).
\]
\item{(c)} If $c$ is a positive constant, then $\widetilde{c}_k$
  satisfies the scaling property
\begin{equation}
\label{eqn:capacityScaling}
\widetilde{c}_k(Y,c\lambda) = c \cdot \widetilde{c}_k(Y,\lambda).
\end{equation}
This follows from the commutative diagram
\[
\begin{CD}
ECH^L(Y,\lambda,0) @>>> ECH(Y,\lambda,0)\\
@V{s}V{\simeq}V @V{s}V{\simeq}V \\
ECH^{cL}(Y,c\lambda,0) @>>> ECH(Y,c\lambda,0),
\end{CD}
\]
where $s$ is the scaling isomorphism \eqref{eqn:scaling}.  And
commutativity of the above diagram is immediate from the definitions.
\item{(d)} One can also define analogues of the full ECH spectrum
  using $ECH(Y,\lambda,\Gamma)$ for $\Gamma\neq 0$.  However
  restricting to $\Gamma$ torsion is necessary to obtain well-defined
  capacities, see Lemma~\ref{lem:dlambda} below.
\end{description}
\end{remark}

\begin{lemma}
\label{lem:monotone1}
  Let $(X,\omega)$ be a weakly exact symplectic cobordism from
  $(Y_+,\lambda_+)$ to $(Y_-,\lambda_-)$.  Assume that the contact
  forms $\lambda_\pm$ are nondegenerate and that $X$ is diffeomorphic
  to a product $[0,1]\times Y$.  Then for every positive integer
  $k$ we have
\[
\widetilde{c}_k(Y_-,\lambda_-) \le \widetilde{c}_k(Y_+,\lambda_+).
\]
\end{lemma}

\begin{proof}
  Fix $L\in\R$ and let $I_\pm$ denote the image of
  $ECH^L(Y_\pm,\lambda_\pm,0)$ in $ECH(Y_\pm,\lambda_\pm,0)$.  We need to
  show that $\dim(I_-)\ge\dim(I_+)$.  By Theorem~\ref{thm:cob}(a) we have
  a commutative diagram
\begin{equation}
\label{eqn:bd}
\begin{CD}
ECH^L(Y_+,\lambda_+,0) @>>> ECH(Y_+,\lambda_+,0) \\
@VV{\Phi^L(X,\omega)}V @VV{\Phi(X,\omega)}V \\
ECH^L(Y_-,\lambda_-,0) @>>> ECH(Y_-,\lambda_-,0).
\end{CD}
\end{equation}
It follows from this diagram that $\Phi(X,\omega)(I_+)\subset I_-$.
By Theorem~\ref{thm:cob}(c) the map $\Phi(X,\omega)$ is an
isomorphism, so $\dim(I_+)\le\dim(I_-)$ as desired.
\end{proof}

We now extend the definition of the full ECH spectrum to arbitrary
(possibly degenerate) contact forms $\lambda$ on $Y$.

\begin{definition}
\label{def:supinf}
Let $(Y,\lambda)$ be any closed contact 3-manifold.  Define
\begin{equation}
\label{eqn:supinf}
\widetilde{c}_k(Y,\lambda) \eqdef
\sup\{\widetilde{c}_k(Y,f_-\lambda)\} =
\inf\{\widetilde{c}_k(Y,f_+\lambda)\},
\end{equation}
where the supremum is over smooth functions $f_-:Y\to(0,1]$ such that
the contact form $f_-\lambda$ is nondegenerate, and the infimum is
over smooth functions $f_+:Y\to[1,\infty)$ such that $f_+\lambda$ is
nondegenerate.
\end{definition}

To confirm that this definition makes sense, we have:

\begin{lemma}
\label{lem:ckpm}
The supremum and infimum in \eqref{eqn:supinf} are equal.
\end{lemma}

\begin{proof}
  We first show that $\sup\{\widetilde{c}_k(Y,f_-\lambda)\} \le
  \inf\{\widetilde{c}_k(Y,f_+\lambda)\}$.  If $f_-,f_+$ are as in
  Definition~\ref{def:supinf}, then
\[
\left([0,1]\times
  Y,d(((1-s)f_-+sf_+)\lambda)\right)
\]
is an exact symplectic cobordism from
  $(Y,f_+\lambda)$ to $(Y,f_-\lambda)$, where $s$ denotes the $[0,1]$
  coordinate.  Thus by Lemma~\ref{lem:monotone1} we have
  $\widetilde{c}_k(Y,f_-\lambda)\le \widetilde{c}_k(Y,f_+\lambda)$.

  We now show that $\sup\{\widetilde{c}_k(Y,f_-\lambda)\} \ge
  \inf\{\widetilde{c}_k(Y,f_+\lambda)\}$.  Fix $\varepsilon>0$.  We
  can find a function $\phi:Y\to(0,\varepsilon)$ such that if $f_+ =
  e^\phi$, then the contact form $f_+\lambda$ is nondegenerate.
  Define $f_-=e^{-\varepsilon}f_+$.  Then by the scaling property
  \eqref{eqn:capacityScaling} we have
\[
\widetilde{c}_k(Y,f_+\lambda) = e^{\varepsilon} \widetilde{c}_k(Y,f_-\lambda).
\]
Thus $\inf\{\widetilde{c}_k(Y,f_+\lambda)\} \le e^{\varepsilon}
\sup\{\widetilde{c}_k(Y,f_-\lambda)\}$.  Now take $\epsilon\to 0$.
\end{proof}

Lemma~\ref{lem:monotone1} then extends to the possibly degenerate case:

\begin{proposition}
\label{prop:monotone2}
Let $(X,\omega)$ be a weakly exact symplectic cobordism from
$(Y_+,\lambda_+)$ to $(Y_-,\lambda_-)$.  Assume that $X$ is
diffeomorphic to a product $[0,1]\times Y$.  Then for every
positive integer $k$ we have
\[
\widetilde{c}_k(Y_-,\lambda_-) \le \widetilde{c}_k(Y_+,\lambda_+).
\]
\end{proposition}

\begin{proof}
  If $f_+$ and $f_-$ are
  functions as in Definition~\ref{def:supinf}, then
\[
(\{(s,y)\in \R\times Y_+ \mid 1\le e^s\le f_+(y)\},d(e^s\lambda_+))
\]
is an exact symplectic cobordism from $(Y_+,f_+\lambda_+)$ to
$(Y_+,\lambda)$, and
\[
(\{(s,y)\in\R\times Y_-\mid f_-(y)\le e^s\le 1\},d(e^s\lambda_-))
\]
is an exact symplectic cobordism from $(Y_-,\lambda_-)$ to
$(Y_-,f_-\lambda_-)$.  Attaching these cobordisms to the positive and
negative boundaries of $X$ defines a subset of the symplectization
completion \eqref{eqn:completion} which is a weakly exact symplectic
cobordism, diffeomorphic to a product, from $(Y_+,f_+\lambda_+)$ to
$(Y_-,f_-\lambda_-)$.  By Lemma~\ref{lem:monotone1} we have
\[
\widetilde{c}_k(Y_-,f_-\lambda_-) \le \widetilde{c}_k(Y_+,f_+\lambda_+).
\]
Taking the supremum over $f_-$ on the left hand side and the infimum
over $f_+$ on the right hand side completes the proof.
\end{proof}

\subsection{Full ECH capacities}
\label{sec:fech}

\begin{definition}
\label{def:capacity}
Let $(X,\omega)$ be a 4-dimensional Liouville domain with boundary
$Y$. If $k$ is a positive integer, define
\[
\widetilde{c}_k(X,\omega) \eqdef \widetilde{c}_k(Y,\lambda),
\]
where $\lambda$ is a contact form on $Y$ with $d\lambda=\omega|_Y$.
We call the numbers $\{\widetilde{c}_k(X,\omega)\}_{k=1,2,\ldots}$ the
{\em full ECH capacities\/} of $(X,\omega)$.
\end{definition}

\begin{lemma}
\label{lem:capacityDefined}
$\widetilde{c}_k(X,\omega)$ does not depend on the choice of contact
form $\lambda$.
\end{lemma}

\begin{proof}
  Let $\lambda'$ be another contact form on $Y$ with $d\lambda'=\omega|_Y$.
We need to show that
\begin{equation}
\label{eqn:capacityDefined}
\widetilde{c}_k(Y,\lambda)=\widetilde{c}_k(Y,\lambda').
\end{equation}
By modifying $X$ slightly as in the proof of
Proposition~\ref{prop:monotone2}, we may assume that $\lambda$ and
$\lambda'$ are nondegenerate.  Equation \eqref{eqn:capacityDefined}
then follows immediately from Definition~\ref{def:fullspectrum} and
Lemma~\ref{lem:dlambda} below.
\end{proof}

\begin{lemma}
\label{lem:dlambda}
Let $Y$ be a closed oriented $3$-manifold.  Let $\lambda, \lambda'$ be
nondegenerate contact forms on $Y$ with $d\lambda=d\lambda'$.  Then
there is an isomorphism $ECH(Y,\lambda,0)\simeq ECH(Y,\lambda',0)$,
which is the direct limit of isomorphisms
\[
ECH^L(Y,\lambda,0) \simeq ECH^L(Y,\lambda',0),
\]
and which respects the $U$ maps.
\end{lemma}

(The part about $U$ maps is not needed here, but will be used in
\S\ref{sec:dbp}.)

\begin{proof}
Let $R$ and $R'$ denote the Reeb vector fields for $\lambda$ and
$\lambda'$ respectively.  Since $d\lambda=d\lambda'$, we have $R'=fR$
for some positive function $f:Y\to\R$.  In particular there is a
canonical bijection between the ECH generators of $\lambda$ and those
of $\lambda'$.

Now define a diffeomorphism
\[
\begin{split}
\phi:\R\times Y &\longrightarrow \R\times Y,\\
(s,y) &\longmapsto (f(y)s,y).
\end{split}
\]
If $J$ is an almost complex structure on $\R\times Y$ as needed to
define the ECH of $\lambda$, then $J'=\phi_*^{-1}\circ J \circ \phi_*$
is an almost complex structure as needed to define the ECH of
$\lambda'$.  The canonical bijection on ECH generators then gives an
isomorphism of chain complexes
\begin{equation}
\label{eqn:bin}
ECC(Y,\lambda,\Gamma,J) \simeq ECC(Y,\lambda',\Gamma,J'),
\end{equation}
because $\phi$ by definition induces a bijection on the relevant
holomorphic curves.  For the same reason, this isomorphism respects
the $U$ maps.

When $\Gamma=0$, the isomorphism \eqref{eqn:bin} further respects the
symplectic action filtrations, because if $\alpha=\{(\alpha_i,m_i)\}$
is an ECH generator with $[\alpha]=0$, then since $\lambda-\lambda'$
is a closed $1$-form on $Y$, by Stokes's theorem we have
$\sum_i m_i\int_{\alpha_i}\lambda=\sum_im_i\int_{\alpha_i}\lambda'$.
\end{proof}

\begin{remark}
\label{rem:c10}
We always have $\widetilde{c}_1(X,\omega)=0$, by
Remark~\ref{rem:spec}(b), because the ECH contact invariant
$[\emptyset] \in ECH(Y,\lambda,0)$ is nonzero by
Theorem~\ref{thm:cob}(b).
\end{remark}

We can now prove a symplectic embedding obstruction, which is a warmup
to Theorem~\ref{thm:main}:

\begin{proposition}
\label{prop:embeddingObstruction}
Let $(X_0,\omega_0)$ and $(X_1,\omega_1)$ be four-dimensional
Liouville domains.  Suppose there is a symplectic embedding
$\varphi:(X_0,\omega_0)\to (\op{int}(X_1),\omega_1)$ such that
$X_1\setminus\op{int}(\varphi(X_0))$ is diffeomorphic to a product
$[0,1]\times Y$.  Then $\widetilde{c}_k(X_0,\omega_0)\le
\widetilde{c}_k(X_1,\omega_1)$ for all positive integers $k$.
\end{proposition}

\begin{proof}
  For $i=0,1$, write $Y_i=\partial X_i$, and let $\lambda_i$ be a
  contact form on $Y_i$ with $d\lambda_i={\omega_i}|_{Y_i}$.  Then
  $(X_1\setminus\op{int}(\varphi(X_0)),\omega_1)$ is a weakly exact
  symplectic cobordism from $(Y_1,\lambda_1)$ to $(Y_0,\lambda_0)$.
  Now apply Proposition~\ref{prop:monotone2}.
\end{proof}

\subsection{The full ECH capacities of an ellipsoid}
\label{sec:fece}

Recall the notation from Proposition~\ref{prop:ellipsoid}.

\begin{proposition}
\label{prop:ellipsoidCapacity}
The full ECH capacities of an ellipsoid are given by
\[
\widetilde{c}_k(E(a,b)) = (a,b)_k.
\]
\end{proposition}

\begin{proof}
  For the contact form on $\partial E(a,b)$ obtained by restricting
  \eqref{eqn:Liouville}, the Reeb vector field is given by
\[
R = 2\pi\left(a^{-1}\frac{\partial}{\partial\theta_1} +
b^{-1}\frac{\partial}{\partial \theta_2}\right),
\]
where $\partial/\partial\theta_j \eqdef x_j\partial/\partial y_j -
y_j\partial/\partial x_j$.

Suppose that the ratio $a/b$ is irrational.  In this case there are
just two embedded Reeb orbits $\gamma_1=(z_2=0)$ and
$\gamma_2=(z_1=0)$.  These are elliptic and nondegenerate and have
action $a$ and $b$ respectively.  In particular $\lambda|_{\partial
  E(a,b)}$ is nondegenerate, and the ECH generators have the form
$\gamma_1^m\gamma_2^n$ where $m,n\in\N$.  Of course these all
correspond to $\Gamma=0$ since $H_1(\partial E(a,b))=0$.  The action
of such a generator is given by
\[
\mc{A}(\gamma_1^m\gamma_2^n) = am+bn.
\]
Since all Reeb orbits are elliptic, all ECH generators have even
grading (see \cite[Prop.\ 1.6(c)]{pfh2}), so the differential on the
ECH chain complex vanishes for any $J$.  (The full calculation of the
grading on the ECH chain complex in this example is given in
\cite[Ex.\ 4.2]{wh}, but we do not need this here.)  Thus the
dimension of the image of $ECH^L(\partial E(a,b),\lambda,0)$ in
$ECH(\partial E(a,b),\lambda,0)$ is
\[
\left|\left\{(m,n)\in \N^2\,\big|\, ma+nb<L\right\}\right|.
\]
The proposition in this case follows immediately.

To prove the proposition when $a/b$ is rational, choose real numbers
$a_-<a<a_+$ and $b_-<b<b_+$ with $a_-/b_-$ and $a_+/b_+$ irrational.  By
Proposition~\ref{prop:embeddingObstruction} we have
\[
(a_-,b_-)_k =
\widetilde{c}_k(E(a_-,b_-)) \le \widetilde{c}_k(E(a,b)) \le
\widetilde{c}_k(E(a_+,b_+)) = (a_+,b_+)_k.
\]
For any given $k$, taking a limit as $a_\pm\to a$ and $b_\pm\to b$
proves that $\widetilde{c}_k(E(a,b))=(a,b)_k$ as claimed.
\end{proof}

If $E(a,b)$ symplectically embeds into the interior of $E(c,d)$, then
Propositions~\ref{prop:embeddingObstruction} and
\ref{prop:ellipsoidCapacity} tell us that
\begin{equation}
\label{eqn:abcd}
(a,b)_k \le (c,d)_k
\end{equation}
for all $k$.  To understand this condition in examples, the following
alternate description of $(a,b)_k$ is useful.  Given $(m,n)\in\N^2$,
let $T_{a/b}(m,n)$ denote the triangle in $\R^2$ whose edges are the
coordinate axes together with the line through $(m,n)$ of slope
$-a/b$.  Then
\[
(a,b)_k = am + bn
\]
where
\[
k = \left|T_{a/b}(m,n)\cap\N^2\right|.
\]

For example, we have $(a,b)_1=0$, as we already knew from
Remark~\ref{rem:c10}.  Next, we have
\[
(a,b)_2 = \left\{\begin{array}{cl} a, & a/b\le 1,\\
b, & a/b \ge 1.
\end{array}\right.
\]
Thus the condition \eqref{eqn:abcd} for $k=2$ recovers the well-known
fact that if $E(a,b)$ symplectically embeds into $E(c,d)$ then
$\min(a,b)\le\min(c,d)$.  Next, assuming $a\ge b$, we have
\begin{equation}
\label{eqn:c3}
(a,b)_3 = \left\{\begin{array}{cl}
2b, & 2\le a/b,\\
a, & 1\le a/b \le 2.
\end{array}
\right.
\end{equation}
Another example is
\begin{equation}
\label{eqn:c6}
(a,b)_6 = \left\{\begin{array}{cl} 
5b, & 5\le a/b,\\
a, & 4\le a/b\le 5,\\
4b, & 3\le a/b\le 4,\\
a+b, & 2\le a/b\le 3,\\
3b, & 3/2\le a/b\le 2,\\
2a, & 1\le a/b \le 3/2.
\end{array}
\right.
\end{equation}

For example, return to the function $f$ defined in \S\ref{sec:eb} that
measures the obstruction to symplectically embedding an ellipsoid into
a ball.  It is computed in \cite{m} that $f(2)=2$ and $f(5)=5/2$.  On
the other hand, equation \eqref{eqn:c3} implies that
$(2,1)_3/(1,1)_3=2$, and equation \eqref{eqn:c6} implies that
$(5,1)_6/(1,1)_6=5/2$.  This is how one confirms that the bound
\eqref{eqn:eb} (which we have already justified) is sharp for $a=2,5$.

\begin{remark}
\label{rem:ellipsoidVolume}
If we write $L=am+bn$, then the triangle $T_{a/b}(m,n)$ has area
$L^2/2ab$, so when $L$ is large,
\[
\left|T_{a/b}(m,n)\cap\N^2\right| = \frac{L^2}{2ab} + O(L).
\]
Note also that $E(a,b)$ has volume $ab/2$.  It follows that
\begin{equation}
\label{eqn:ellipsoidVolume}
\lim_{k\to\infty}\frac{\widetilde{c}_k(E(a,b))^2}{k} = 4\op{vol}(E(a,b)).
\end{equation}
In particular, the condition \eqref{eqn:abcd} for $k$ large simply
tells us that the volume of $E(a,b)$ is less than or equal to that of
$E(c,d)$.  (But the equality in \eqref{eqn:ellipsoidVolume} only holds
in the limit, so that for given $(a,b)$ and $(c,d)$, taking suitable
small values of $k$ often gives stronger conditions.)
\end{remark}

\section{Distinguished ECH spectrum and capacities}
\label{sec:distinguished}

We now define modified versions of the full ECH spectrum and full ECH capacities
which give obstructions to symplectic embeddings for non-product cobordisms.

\subsection{Definitions and basic properties}
\label{sec:dbp}

\begin{definition}
\label{def:csigma}
  If $\lambda$ is a nondegenerate contact form on a closed oriented
  three-manifold $Y$, and if $0\neq\sigma\in ECH(Y,\lambda,\Gamma)$, define
  $c_\sigma(Y,\lambda)$ to be the infimum over $L\in\R$ such that
  $\sigma$ is contained in the image of the map $ECH^L(Y,\lambda,\Gamma)\to
  ECH(Y,\lambda,\Gamma)$.  As in \S\ref{sec:ECHspectrum}, if $\lambda$ is
  degenerate, define
\[
c_\sigma(Y,\lambda) \eqdef \sup \{c_\sigma(Y,f\lambda)\},
\]
where the supremum is over functions $f:Y\to(0,1]$ such that
$f\lambda$ is nondegenerate.  Note that this definition makes sense
because $ECH(Y,f\lambda,\Gamma)$ does not depend on $f$.  (The
cobordism maps \eqref{eqn:cc2} for product cobordisms define a
canonical isomorphism $ECH(Y,f\lambda,\Gamma)=ECH(Y,f'\lambda,\Gamma)$
whenever $f\lambda$ and $f'\lambda$ are nondegenerate.)
\end{definition}

It follows from Lemma~\ref{lem:amb} below that if $\sigma\in
ECH(Y,\lambda,0)$, then $c_\sigma(Y,\lambda)$ is one of the numbers in
the full ECH spectrum $\{\widetilde{c}_k(Y,\lambda)\}$.

\begin{lemma}
\label{lem:ctilde}
Let $(X,\omega)$ be a weakly exact symplectic cobordism from
$(Y_+,\lambda_+)$ to $(Y_-,\lambda_-)$, where $\lambda_\pm$ are
nondegenerate. Let $\sigma\in ECH(Y_+,\lambda_+,0)$.  Then
\[
c_{\sigma}(Y_+,\lambda_+) \ge c_{\Phi(X,\omega)(\sigma)}(Y_-,\lambda_-).
\]
\end{lemma}

\begin{proof}
Let $L\in\R$.  Suppose $\sigma$ is
  in the image of the map $ECH^L(Y_+,\lambda_+,0)\to
  ECH(Y_+,\lambda_+,0)$.  Then it follows from the diagram
  \eqref{eqn:bd} that $\Phi(X,\omega)(\sigma)$ is in the image of the
  map $ECH^L(Y_-,\lambda_-,0)\to ECH(Y_-,\lambda_-,0)$.
\end{proof}

\begin{definition}
If $(Y,\lambda)$ is a closed connected contact three-manifold with
$c(\xi)\neq 0$, and if $k$ is a nonnegative integer, define
\begin{equation}
\label{eqn:desc}
c_k(Y,\lambda) \eqdef \min\left\{c_\sigma(Y,\lambda) \,\big|\, \sigma\in
ECH(Y,\lambda,0), \;
U^k\sigma=[\emptyset]\right\}.
\end{equation}
More generally, if $(Y,\lambda)$ is a closed contact three-manifold
with connected components $Y_1,\ldots,Y_n$, and if $c(\xi)\neq 0$,
define
\[
\begin{split}
c_k(Y,\lambda) \eqdef \min\big\{c_\sigma(Y,\sigma) \;\big|\; &
  \sigma\in ECH(Y,\lambda,0),\\
&  U_{i_1}\cdots U_{i_k}\sigma=[\emptyset] \;\;\; \forall
i_1,\ldots,i_k\in\{1,\ldots,n\}\big\}.
\end{split}
\]
The sequence $\{c_k(Y,\lambda)\}_{k=0,1,\ldots}$ is called the {\em
  (distinguished) ECH spectrum\/} of $(Y,\lambda)$.
\end{definition}

\begin{remark}
\label{rem:torsion}
\begin{description}
\item{(a)}
Any choice of chain map used to define the $U$ map decreases the
symplectic action, for the same reason that the differential does, see
\S\ref{sec:filtered}.  It follows that
\[
0 = c_0(Y,\lambda) < c_1(Y,\lambda) \le c_2(Y,\lambda) \le \cdots \le \infty.
\]
Here $c_k(Y,\lambda)=c_{k+1}(Y,\lambda)<\infty$ is possible when
$\lambda$ is degenerate.
\item{(b)} We have $c_k(Y,\lambda)<\infty$ for all $k$ only if
  $c_1(\xi)\in H^2(Y;\Z)$ is torsion.  Proof: Without loss of
  generality $Y$ is connected.  Recall from Remark~\ref{rem:spec} that
  if $c_1(\xi)$ is not torsion then $ECH(Y,\lambda,0)$ is finitely
  generated.  But if $\sigma\in ECH(Y,\lambda,0)$ and
  $U^k\sigma=[\emptyset]$ then $\dim(ECH(Y,\lambda,0))>k$, because it
  follows from $U[\emptyset]=0$ that the classes
  $\sigma,U\sigma,\ldots,U^k\sigma$ are linearly independent.
\end{description}
\end{remark}

In simple examples the distinguished ECH spectrum is related to the
full ECH spectrum defined previously as follows.  Recall from
\cite{eliashberg} that there is a unique tight contact structure on
$S^3$, which is the one induced by a Liouville domain with boundary
diffeomorphic to $S^3$.

\begin{proposition}
\label{prop:dss3}
If $Y$ is diffeomorphic to $S^3$ and if $\Ker(\lambda)$ is
  the tight contact structure on $Y$, then
\[
c_k(Y,\lambda) = \widetilde{c}_{k+1}(Y,\lambda).
\]
\end{proposition}

We also have:

\begin{proposition}
\label{prop:dsdu}
If $(Y_i,\lambda_i)$ are closed contact 3-manifolds with nonvanishing
ECH contact invariant for $i=1,\ldots,n$, then
\begin{equation}
\label{eqn:dsdu}
c_k\left(\coprod_{i=1}^n(Y_i,\lambda_i)\right) = \max
\left\{\sum_{i=1}^nc_{k_i}(Y_i,\lambda_i) \;\bigg|\;
  \sum_{i=1}^n k_i=k\right\}.
\end{equation}
\end{proposition}

The proofs of the above two propositions require an algebraic
digression which is deferred to \S\ref{sec:interlude}.

\begin{proposition}
\label{prop:cobtilde}
If $(X,\omega)$ is a weakly exact symplectic cobordism from
$(Y_+,\lambda_+)$ to $(Y_-,\lambda_-)$, then
\[
c_k(Y_+,\lambda_+) \ge c_k(Y_-,\lambda_-)
\]
for each nonnegative integer $k$.
\end{proposition}

\begin{proof}
  By the approximation argument in the proof of
  Proposition~\ref{prop:monotone2}, we may assume that $\lambda_+$ and
  $\lambda_-$ are nondegenerate.

  Let $Y_\pm^1,\ldots,Y_\pm^{n_\pm}$ denote the connected components
  of $Y_\pm$. Let $\sigma_+\in ECH(Y_+,\lambda_+,0)$ be a class with
  $U_{i_1}\cdots U_{i_k}\sigma=[\emptyset]$ for all
  $i_1,\ldots,i_k\in\{1,\ldots,n_+\}$.  Let $\sigma_-\eqdef
  \Phi(X,\omega)(\sigma_+)\in ECH(Y_-,\lambda_-,0)$.  Since each
  component of the cobordism $X$ has at least one positive boundary
  component, it follows from Theorem~\ref{thm:cob}(b),(d) that
  $U_{i_1}\cdots U_{i_k}\sigma_-=[\emptyset]$ for all
  $i_1,\ldots,i_k\in\{1,\ldots,n_-\}$.  By Lemma~\ref{lem:ctilde} we
  have $c_{\sigma_+}(Y_+,\lambda_+) \ge c_{\sigma_-}(Y_-,\lambda_-)$.
\end{proof}

\begin{definition}
  By analogy with Definition~\ref{def:capacity}, if $(X,\omega)$ is a
  4-dimensional Liouville domain with boundary $Y$, and if $k$ is a
  nonnegative integer, define
\[
c_k(X,\omega) \eqdef c_k(Y,\lambda),
\]
where $\lambda$ is a contact form on $Y$ with $d\lambda=\omega|_Y$.
Lemma~\ref{lem:dlambda} shows that this does not depend on the choice
of contact form $\lambda$, just like the full ECH capacities.  The
numbers $c_k(X,\omega)$ are called the {\em (distinguished) ECH
  capacities\/} of $(X,\omega)$.
\end{definition}

We can now prove the main symplectic embedding obstruction:

\begin{proof}[Proof of Theorem~\ref{thm:main}.]
  For $i=0,1$, let $Y_i=\partial X_i$ and let $\lambda_i$ be a contact
  form on $Y_i$ with $d\lambda_i={\omega_i}|_{X_i}$.  Then $X_1$ minus
  the interior of the image of $X_0$ defines a weakly exact symplectic
  cobordism from $(Y_1,\lambda_1)$ to $(Y_0,\lambda_0)$.  By
  Proposition~\ref{prop:cobtilde}, $c_k(X_0,\omega_0) \le
  c_k(X_1,\omega_1)$.  But in fact the inequality is strict when
  $c_k(X_0,\omega_0)<\infty$, because the embedding sends $X_0$ into
  the interior of $X_1$, so we can extend the embedding over
  $[0,\epsilon]\times Y_0$ in the symplectization completion
  \eqref{eqn:completion} of $X_0$ for some $\epsilon>0$.  The above
  argument together with the scaling isomorphism \eqref{eqn:scaling}
  then shows that $e^{\epsilon}c_k(X_0,\omega_0)\le
  c_k(X_1,\omega_1)$.
\end{proof}

\subsection{More general domains}
\label{sec:mgd}

We now explain how to extend the definition of the (distinguished) ECH
capacities to some more general spaces.

\begin{definition}
\label{def:extendck}
  Let $(X,\omega)$ be a  subset of a symplectic
  four-manifold.  If $k$ is a positive integer, define
\[
c_k(X,\omega) \eqdef \sup\{c_k(X_-,\omega)\},
\]
where the supremum is over subsets $X_-\subset \op{int}(X)$ such that
$(X_-,\omega)$ is a four-dimensional Liouville domain.
\end{definition}

By definition, $c_k(X,\omega)$ depends only on the symplectic form on
$\op{int}(X)$, and not on the symplectic four-manifold of which $X$ is
a subset.  If $(X,\omega)$ is already a four-dimensional Liouville
domain, then by Theorem~\ref{thm:main} the above definition of
$c_k(X,\omega)$ agrees with the previous one.

\begin{remark}
  One could also try to define the full ECH capacities of a 
  subset of a symplectic four-manifold as in
  Definition~\ref{def:extendck}.  However it is not clear if this
  would agree with the previous definition for Liouville domains,
  because of the extra assumption in
  Proposition~\ref{prop:embeddingObstruction}.  This is another way in
  which distinguished ECH capacities work better than full ECH
  capacities.
\end{remark}

We now have the following extension of Theorem~\ref{thm:main}:

\begin{proposition}
\label{prop:extend}
Suppose that $(X_i,\omega_i)$ is a
 subset of a symplectic four-manifold for $i=0,1$.  If
there is a symplectic embedding $\varphi:X_0\to\op{int}(X_1)$, then
$c_k(X_0,\omega_0)\le c_k(X_1,\omega_1)$ for all $k$.
\end{proposition}

\begin{proof}
  This is a tautology. Let $X_-$ be a subset of $\op{int}(X_0)$ such
  that $(X_-,\omega_0)$ is a four-dimensional Liouville domain.  Then
  $\varphi$ restricts to a symplectic embedding of $X_-$ into
  $\op{int}(X_1)$, so by Definition~\ref{def:extendck},
\[
c_k(X_-,\omega_0) \le c_k(X_1,\omega_1).
\]
Taking the supremum over $X_-$ on the left hand side completes the proof.
\end{proof}

Note also that Proposition~\ref{prop:du} extends to the case when each
$(X_i,\omega_i)$ is a  subset of a symplectic
four-manifold.

\section{Algebraic interlude}
\label{sec:interlude}

The goal of this section is to prove Propositions~\ref{prop:dss3} and
\ref{prop:dsdu}.  To simplify the notation, in this section  write
$H(Y,\lambda)\eqdef ECH(Y,\lambda,0)$, and let $H^L(Y,\lambda)$
denote the image of $ECH^L(Y,\lambda,0)$ in $ECH(Y,\lambda,0)$.  Also
write $C_*(Y,\lambda,J)\eqdef ECC(Y,\lambda,0,J)$, and let
$C^*(Y,\lambda,J)$ denote the dual chain complex
$\Hom(C_*(Y,\lambda,J),\Z/2)$.

\begin{definition}
  Let $\lambda$ be a nondegenerate contact form on a closed
  $3$-manifold $Y$.  A basis $\{\sigma_k\}_{k=1,2,\ldots}$ for
  $H(Y,\lambda)$ is {\em action-minimizing\/} if
\begin{equation}
\label{eqn:csigmak}
c_{\sigma_k}(Y,\lambda)= \widetilde{c}_k(Y,\lambda)
\end{equation}
for all $k$.
\end{definition}

\begin{lemma}
\label{lem:amb}
Let $\lambda$ be a nondegenerate contact form on a closed $3$-manifold
$Y$. Then:
\begin{description}
\item{(a)}
There exists an action-minimizing basis for
$H(Y,\lambda)$.
\item{(b)} If $\{\sigma_k\}$ is an action-minimizing basis for
  $H(Y,\lambda)$, and if $0\neq\sigma=\sum_ja_j\sigma_j\in
  H(Y,\lambda)$, then
\begin{equation}
\label{eqn:usefulbasis}
c_\sigma(Y,\lambda)=\widetilde{c}_k(Y,\lambda)
\end{equation}
where $k$ is the largest integer such that $a_k\neq 0$.
\end{description}
\end{lemma}

\begin{proof}
  (a) To construct an action-minimizing basis, increase $L$ starting
  from $0$, and whenever the dimension of $H^L(Y,\lambda)$
  jumps, add new basis elements to span the rest of it.  More
  precisely, there is a discrete set of nonnegative real numbers $L$
  such that
\[
\dim(H^{L+\epsilon}(Y,\lambda))>\dim(H^L(Y,\lambda))
\]
for all $\epsilon>0$.  Denote these real numbers by $0\le L_1 < L_2 <
\cdots$.  There are then integers $0=k_0<k_1<k_2<\cdots$ such that
\begin{equation}
\label{eqn:cki}
k_{i-1}<k\le k_i \;\Longrightarrow\; \widetilde{c}_k(Y,\lambda)=L_i.
\end{equation}
Now define a
basis by taking $\{\sigma_k\mid k_{i-1}<k\le k_i\}$ to be elements of
$H^{L_{i+1}}(Y,\lambda)$ that project to a basis for
$H^{L_{i+1}}(Y,\lambda)/H^{L_i}(Y,\lambda)$.  Then equation
\eqref{eqn:csigmak} follows from the construction.

To prepare for the proof of (b), note also that conversely, by
\eqref{eqn:cki}, any action-minimizing basis is obtained by the above
construction.

(b) Continuing the notation from the proof of part (a), we have
$\widetilde{c}_k(Y,\lambda)=L_i$ for some $i$.  By equation
\eqref{eqn:csigmak}, $\sigma\in H^L(Y,\lambda)$ whenever $L>L_i$, so
$c_\sigma(Y,\lambda)\le L_i$.  To prove the reverse inequality,
suppose to get a contradiction that $\sigma\in H^{L_i}(Y,\lambda)$.
Let $\sigma'$ denote the contribution to $\sigma$ from basis elements
$\sigma_j$ with $c_{\sigma_j}(Y,\lambda)<L_i$.  Then $\sigma'\in
H^{L_i}(Y,\lambda)$, so $\sigma-\sigma'\in H^{L_i}(Y,\lambda)$ as
well.  Now $\sigma-\sigma'$ is a linear combination of the basis
elements $\{\sigma_k\mid k_{i-1}<k\le k_i\}$.  Since the latter are
linearly independent in $H^{L_{i+1}}(Y,\lambda)/H^{L_i}(Y,\lambda)$,
it follows that $\sigma-\sigma'=0$, which is the desired
contradiction.
\end{proof}

\begin{remark}
\label{rem:careful}
  One has to be careful in the proof of
  Lemma~\ref{lem:amb}(b), because the equality
\begin{equation}
\label{eqn:careful}
c_{\sigma_1+\cdots+\sigma_n}(Y,\lambda) = \max\{c_{\sigma_i}(Y,\lambda)\mid
i=1,\ldots,n\}
\end{equation}
does not always hold for linearly independent elements
$\sigma_1,\ldots,\sigma_n$ of $H(Y,\lambda)$.  However
\eqref{eqn:careful} does hold if the maximum on the right hand side is
realized by a unique $i\in\{1,\ldots,n\}$, or if all of the classes
$\sigma_1,\ldots,\sigma_n$ have (definite and) distinct gradings.
\end{remark}

\begin{proof}[Proof of Proposition~\ref{prop:dss3}.]
  By the usual approximation arguments we may assume that $\lambda$ is
  nondegenerate.  Since $Y$ is a homology sphere, the relative grading
  on ECH has a canonical refinement to an absolute $\Z$-grading in
  which the empty set of Reeb orbits has grading zero.  With this
  grading convention, the ECH with $\Z/2$-coefficients is given by
\[
ECH_*(Y,\lambda,0) = \left\{\begin{array}{cl} \Z/2, & *=0,2,\ldots,\\\
0, & \mbox{otherwise.}\end{array}\right.
\]
In addition, $U:ECH_*(Y,\lambda,0)\to ECH_{*-2}(Y,\lambda,0)$ is an
isomorphism whenever $*\neq 0$.  These facts follow from the
isomorphism \eqref{eqn:echswf}, together with the computation of the
Seiberg-Witten Floer homology of $S^3$ in \cite{km}.  Finally,
$[\emptyset]$ generates $ECH_0(Y,\lambda,0)$.  This follows from the
above facts, or from direct computations for a standard tight contact
form on $S^3$, see \cite[Ex.\ 4.2]{wh}.

Now let $\sigma_k$ denote the generator of $ECH_{2k}(Y,\lambda,0)$.
Since the $U$ map decreases symplectic action we have
\begin{equation}
\label{eqn:Ud}
0 = c_{\sigma_0}(Y,\lambda) < c_{\sigma_1}(Y,\lambda) < \cdots <\infty.
\end{equation}
It follows from \eqref{eqn:Ud} and Remark~\ref{rem:careful} that
$c_{\sigma_k}(Y,\lambda)=\widetilde{c}_{k+1}(Y,\lambda)$.  Now a class
$\sigma=\sum_ja_j\sigma_j$ satisfies $U^k\sigma=[\emptyset]$ if and
only if $a_k=1$ and $a_j=0$ for $j>k$.  By Lemma~\ref{lem:amb}(b),
each such class $\sigma$ satisfies $c_\sigma(Y,\lambda) =
\widetilde{c}_{k+1}(Y,\lambda)$.
\end{proof}

Before continuing, we need to recall the following elementary fact:

\begin{lemma}
\label{lem:simple}
Let $(C_*,\partial)$ be a chain complex over a field ${\mathbb F}$,
and let $C_*'\subset C_*$ be a subcomplex.  Suppose
$\alpha_1,\ldots,\alpha_n\in H_*(C_*)$ are linearly independent in
$H_*(C_*)/H_*(C_*')$, and let $y_1,\ldots,y_n\in {\mathbb F}$.  Then
there exists a cocycle $\zeta\in \Hom(C_*,{\mathbb F})$ which
annihilates $C_*'$ and sends $\alpha_i\mapsto y_i$ for each $i$.
\end{lemma}

\begin{proof}
  Let $x_i\in C_*$ be a cycle representing the homology class
  $\alpha_i$.  By hypothesis, $x_1,\ldots,x_n$ project to linearly
  independent elements of $C_*/(C_*' + \partial(C_*))$.  Hence there
  is a linear map $\zeta:C_*\to {\mathbb F}$ sending $x_i\mapsto y_i$
  for each $i$ and annihilating the subspace $C_*'+\partial(C_*)$.
  This is the desired cocycle.
\end{proof}

\begin{proof}[Proof of Proposition~\ref{prop:dsdu}.]
  By the usual approximation argument, we may assume that the contact
  forms $\lambda_i$ are nondegenerate.  We can also assume that each
  $Y_i$ is connected.  We now proceed in three steps.

  {\em Step 1.\/} We first show that the left hand side of
  \eqref{eqn:dsdu} is less than or equal to the right hand side.  We
  can assume that the right hand side is finite.  For each
  $i=1,\ldots,n$ and $j\ge 0$ with $c_j(Y_i,\lambda_i)<\infty$, choose
  a class $\sigma_{i,j}\in H(Y_i,\lambda_i)$ with
  $U^j\sigma_{i,j}=[\emptyset]$, such that $\sigma_{i,j}\in
  H^L(Y_i,\lambda_i)$ whenever $L>c_j(Y_i,\lambda_i)$.  Recalling the
  identification \eqref{eqn:tensor}, define a class
\[
\sigma \eqdef
\sum_{j_1+\cdots+j_n=k}\sigma_{1,j_1}\tensor\cdots\tensor
\sigma_{n,j_n} \in H\left(\coprod_{i=1}^n(Y_i,\lambda_i)\right).
\]
Since symplectic action is additive under tensor product, $\sigma\in
H^L\left(\coprod_i(Y_i,\lambda_i)\right)$ whenever $L$ is greater than
the right hand side of \eqref{eqn:dsdu}.  So we just need to show that
$U_{i_1}\cdots U_{i_k}\sigma=[\emptyset]$ for all
$i_1,\ldots,i_k\in\{1,\ldots,n\}$.  Equivalently, since the different
maps $U_i$ commute, we need to show that if $\sum_{i=1}^nk_i=k$ then
\[
U_1^{k_1}\cdots U_n^{k_n}\sigma=[\emptyset].
\]
To prove this last statement, observe that if $\sum_{i=1}^nj_i=k$ then
\[
U_1^{k_1}\cdots U_n^{k_n}(\sigma_{1,j_1}\tensor\cdots\tensor\sigma_{n,j_n})
= \left\{\begin{array}{cl} [\emptyset], &
    (j_1,\ldots,j_n)=(k_1,\ldots,k_n),\\
0, & \mbox{otherwise}.
\end{array}
\right.
\]
This is because if $(j_1,\ldots,j_n)\neq (k_1,\ldots,k_n)$ then
$k_i>j_i$ for some $i$, so that
\[
U_i^{k_i}\sigma_{i,j_i} = U_i^{k_i-j_i}[\emptyset] = 0,
\]
where the last equality holds since $U_i$ decreases symplectic action.

{\em Step 2.\/} We claim now that
\begin{equation}
\label{eqn:cpl}
H^L\left(\coprod_{i=1}^n(Y_i,\lambda_i)\right) 
=
\op{span}\left\{\bigotimes_{i=1}^n H^{L_i}(Y_i,\lambda_i) \;\bigg|\;
  \sum_{i=1}^n L_i \le L \right\}.
\end{equation}

To prove this, for each $i=1,\ldots,n$, let
$\{\sigma_{i,j}\}_{j=1,2,\ldots}$ be an action-minimizing basis for
$H(Y_i,\lambda_i)$.  By Lemma~\ref{lem:amb}(b), for each $i$ and $L_i$
we have
\[
H^{L_i}(Y_i,\lambda_i) = \op{span}\{\sigma_{i,j}\mid
c_{\sigma_{i,j}}(Y_i,\lambda_i) < L_i\}.
\]
Thus equation \eqref{eqn:cpl} is equivalent to
\begin{equation}
\label{eqn:cpl2}
H^L\left(\coprod_{i=1}^n(Y_i,\lambda_i)\right) 
=
\op{span}\left\{\sigma_{1,j_1}\tensor\cdots\tensor\sigma_{n,j_n}
  \;\bigg|\; \sum_{i=1}^n c_{\sigma_{i,j_i}}(Y_i,\lambda_i) < L
\right\}.
\end{equation}

The right hand side of \eqref{eqn:cpl2} is a subset of the left, as in
Step 1, because in the identification \eqref{eqn:tensorchain} the
symplectic action is additive under tensor product. To prove the
reverse inclusion, consider a class
\begin{equation}
\label{eqn:cac}
\sigma = \sum_{j_1,\ldots,j_n}a_{j_1,\ldots,j_n}
\sigma_{1,j_1}\tensor\cdots\tensor\sigma_{n,j_n} \in
H\left(\coprod_{i=1}^n(Y_i,\lambda_i)\right).
\end{equation}
Let
\[
L' \eqdef \max\left\{\sum_{i=1}^n c_{\sigma_{i,j_i}}(Y_i,\lambda_i) \;\bigg|\;
  a_{j_1,\ldots,j_n}\neq 0\right\}.
\]
We need to show that
$\sigma\notin H^{L'}\left(\coprod_i(Y_i,\lambda_i)\right)$.

To do so, choose $(j_1,\ldots,j_n)$ with $a_{j_1,\ldots,j_n}\neq 0$
and $c_{\sigma_{i,j_i}}(Y_i,\lambda_i)=L_i$ where $\sum_{i=1}^nL_i =
L'$.  Choose an almost complex structure $J_i$ on $\R\times Y_i$ as
needed to define the ECH of $\lambda_i$.  By Lemmas~\ref{lem:amb}(b)
and \ref{lem:simple}, there is a cocycle $\zeta_i\in
C^*(Y_i,\lambda_i,J_i)$ sending $\sigma_{i,j_i}\mapsto 1$,
annihilating all other basis elements $\sigma_{i,j}$ with
$c_{\sigma_{i,j}}(Y_i,\lambda_i)=L_i$, and annihilating all ECH
generators with action less than $L_i$.  Then
\[
\zeta_1\tensor\cdots\tensor\zeta_n\in
C^*\left(\coprod_{i=1}^n(Y_i,\lambda_i,J_i)\right)
\]
sends $\sigma\mapsto 1$ and annihilates
$H^{L'}\left(\coprod_i(Y_i,\lambda_i)\right)$.  Therefore $\sigma\notin
H^{L'}\left(\coprod_i(Y_i,\lambda_i)\right)$.

{\em Step 3.\/} We now show that the left hand side of
\eqref{eqn:dsdu} is greater than or equal to the right hand side.  We
need to show that if $\sum_{i=1}^nk_i=k$ then
\[
c_k\left(\coprod_{i=1}^n(Y_i,\lambda_i)\right) \ge
\sum_{i=1}^nc_{k_i}(Y_i,\lambda_i).
\]
To do so, let $L\eqdef\sum_{i=1}^n c_{k_i}(Y_i,\lambda_i)$.  We will show
that if $\sigma\in H^L\left(\coprod_i(Y_i,\lambda_i)\right)$, then
$U_1^{k_1}\cdots U_n^{k_n}\sigma\neq[\emptyset]$.

Expand $\sigma$ as in \eqref{eqn:cac}.  By Step 2,
\begin{equation}
\label{eqn:ub}
a_{j_1,\ldots,j_n}\neq 0 \;\Longrightarrow\;
\sum_{i=1}^nc_{\sigma_{j_i}}(Y_i,\lambda_i) < L.
\end{equation}

Next, for each $i=1,\ldots,n$, we can choose
$\zeta_i\in\Hom(H(Y_i,\lambda_i),\Z/2)$ with the following two
properties:
\begin{description}
\item{(i)}
$\zeta_i([\emptyset])=1$.
\item{(ii)} $\zeta_i$ annihilates
  $U^{k_i}\left(H^{c_{k_i}(Y_i,\lambda_i)}(Y_i,\lambda_i)\right)$.
\end{description}
Now let
\[
\zeta = \zeta_1\tensor\cdots\tensor\zeta_n \in
\Hom\left(H\left(\coprod_i(Y_i,\lambda_i)\right),\Z/2\right).
\]
By property (i) we have $\zeta([\emptyset])=1$.  On the other hand,
\[
\begin{split}
\zeta\left(U_1^{k_1}\cdots U_n^{k_n}\sigma\right) & = \left(\zeta_1\circ
U_1^{k_1}\right)\tensor\cdots\tensor \left(\zeta_n\circ U_n^{k_n}\right)\sigma\\
&= \sum_{j_1,\ldots,j_n}a_{j_1,\ldots,j_n} \prod_{i=1}^n
\zeta_i\left(U_i^{k_i}\sigma_{i,j_i}\right)\\
&= 0,
\end{split}
\]
where the last equality follows from \eqref{eqn:ub} and (ii).  Thus
$U_1^{k_1}\cdots U_n^{k_n}\sigma\neq[\emptyset]$ as desired.
\end{proof}

\section{The 3-torus}
\label{sec:t3}

We now compute the distinguished ECH spectrum of the 3-torus with various
contact forms.

\subsection{Distinguished ECH spectrum of the standard 3-torus}
\label{sec:dest3}

Consider the 3-torus
\begin{equation}
\label{eqn:t3}
Y=T^3=(\R/2\pi\Z)_\theta\times(\R^2/\Z^2)_{x,y}
\end{equation}
with the standard contact form
\begin{equation}
\label{eqn:stdt3}
\lambda=\cos\theta\,dx + \sin\theta\,dy.
\end{equation}
The ECH of this example was studied in detail in \cite{t3}.  Using
these results, we can now compute the distinguished ECH spectrum:

\begin{proposition}
\label{prop:standardt3}
If $k$ is a nonnegative integer then
\begin{equation}
\label{eqn:T3spectrum}
c_k(T^3,\lambda) = \min\left\{\ell(\Lambda) \,\big|\,
  |P_\Lambda\cap\Z^2|=k+1 \right\}.
\end{equation}
Here the minimum is over convex polygons $\Lambda$ in $\R^2$ with
vertices in $\Z^2$, and $P_\Lambda$ denotes the closed region bounded
by $\Lambda$.  Also $\ell(\Lambda)$ denotes the Euclidean length of $\Lambda$.
\end{proposition}

\begin{proof}
The proof has three steps.

{\em Step 1.\/} We first review what we need to know about the ECH of
$T^3$.  The relative grading on $ECH_*(T^3,\lambda,0)$ has a canonical
refinement to an absolute $\Z$-grading in which the empty set has
grading $0$.  With this convention, we have (by \cite{t3}, or
using the isomorphism \eqref{eqn:echswf} and \cite[Prop.\ 3.10.1]{km})
\begin{equation}
\label{eqn:echt3}
ECH_*(T^3,\lambda,0) \simeq \left\{\begin{array}{cl}(\Z/2)^3, &
    *\ge 0,\\
    0, & \mbox{otherwise.}\end{array}\right.
\end{equation}
In addition, the map
\[
U: ECH_*(T^3,\lambda,0) \longrightarrow ECH_{*-2}(T^3,\lambda,0)
\]
is an isomorphism whenever $*\ge 2$.  Finally, the contact invariant
$[\emptyset]$ is nonzero (by \cite{t3}, or because $(T^3,\lambda)$ is
the boundary of a Liouville domain).

We also need to know a bit about the ECH chain complex.  The Reeb
vector field is given by
\[
R = \cos\theta\frac{\partial}{\partial x} +
\sin\theta\frac{\partial}{\partial y}.
\]
It follows that for every pair of relatively prime integers $(m,n)$
there is a Morse-Bott circle of embedded Reeb orbits $\mc{O}_{m,n}$
sweeping out $\{\theta\}\times(\R^2/\Z^2)$ where
\begin{equation}
\label{eqn:theta}
\cos\theta=\frac{m}{\sqrt{m^2+n^2}},\quad\quad
\sin\theta=\frac{n}{\sqrt{m^2+n^2}}.
\end{equation}
Each Reeb orbit $\gamma\in\mc{O}_{m,n}$ has symplectic action
\begin{equation}
\label{eqn:aomn}
\mc{A}(\gamma) = \sqrt{m^2+n^2}.
\end{equation}
There are no other embedded Reeb orbits.

Fix $L\in\R$.  For any $\epsilon>0$, we can perturb the contact form
$\lambda$ to $f\lambda$ where $f:Y\to[1-\epsilon,1]$, such that each
Morse-Bott circle $\mc{O}_{m,n}$ with $\sqrt{m^2+n^2}<L$ splits into
an elliptic orbit $e_{m,n}$ and a hyperbolic orbit $h_{m,n}$, and
these are the only embedded Reeb orbits with action less than $L$.  As
in \cite[\S11.3]{t3}, a generator $\alpha$ of the ECH chain complex
for $f\lambda$ with action less than $L$ and with $\Gamma=0$ then
corresponds to a convex lattice polygon $\Lambda_\alpha$, modulo
translation, in which each edge is labeled `$e$' or `$h$'.  Note here
that $2$-gons and $0$-gons are allowed, with the latter corresponding
to the empty set of Reeb orbits.

By \eqref{eqn:aomn}, the action of a generator $\alpha$ as above is
given by
\begin{equation}
\label{eqn:at3}
\mc{A}(\alpha) = \ell(\Lambda_\alpha) - O(\epsilon).
\end{equation}
Furthermore, it is shown in \cite[\S11.3]{t3} that with the above
grading conventions, the grading of the generator $\alpha$ is given by
\begin{equation}
\label{eqn:T3grading}
I(\alpha) = 2(|P_{\Lambda_\alpha}\cap\Z^2|-1) - \#h(\alpha),
\end{equation}
where $\#h(\alpha)$ denotes the number of edges of $\Lambda_\alpha$
that are labeled `$h$'.

{\em Step 2.\/} We now prove that the left hand side of
\eqref{eqn:T3spectrum} is less than or equal to the right hand side.

Fix a nonnegative integer $k$.  Let $\Lambda_0$ be a length-minimizing
convex polygon with $|P_{\Lambda_0}\cap\Z^2|=k+1$.  Let $\alpha_0$
denote the ECH generator consisting of the polygon $\Lambda_0$ with
all edges labeled `$e$'.  (Assume that $L$ above is chosen
sufficiently large with respect to $k$ so that this is defined.)  The
differential on the ECH chain complex in action less than $L$ for
suitable perturbation function $f$ and almost complex structure $J$ is
computed in \cite{t3}: roughly speaking, the differential of a
generator is the sum over all ways of ``rounding a corner'' and
``locally losing one `$h$'''.  Since the generator $\alpha_0$ has no
`$h$' labels, it follows immediately that $\partial\alpha_0=0$.  In
addition, it follows from the computation of the $U$ map in
\cite[\S12.1.4]{t3} that the chain map $U$ applied to a generator with
all edges labeled `$e$' is obtained by rounding a distinguished corner
(depending on the choice of point $z\in Y$ used to define the chain
map $U$) and leaving all edges labeled `$e$'.  It follows that
$U^k\alpha_0=\emptyset$.  Thus $[\alpha_0]$ is a class in $ECH$ with
$U^k[\alpha_0]=[\emptyset]$, so
\[
c_k(T^3,f\lambda) \le \mc{A}(\alpha_0) =
\ell(\Lambda_0)-O(\epsilon).
\]
Taking $\epsilon\to0$ proves the desired inequality.

{\em Step 3.\/} We now prove that the left hand side of
\eqref{eqn:T3spectrum} is greater than or equal to the right hand
side.

Let $\sigma\in ECH(T^3,f\lambda,0)$ be a class with
$U^k\sigma=[\emptyset]$.  Since $U$ is an isomorphism in grading $\ge
2$, it follows that $\sigma=[\alpha_0]+\sigma'$ where $\sigma'$ is a
sum of classes of grading less than $2k$.  Thus by Remark~\ref{rem:careful},
\begin{equation}
\label{eqn:pi}
c_\sigma(T^3,f\lambda)
=\max(c_{[\alpha_0]}(T^3,f\lambda),c_{\sigma'}(T^3,f\lambda)) \ge
c_{[\alpha_0]}(T^3,f\lambda).
\end{equation}

Next we observe that
\begin{description}
\item{(*)} $\ell(\Lambda_0)$ is (up to $O(\epsilon)$ error) the
  minimum of $\mc{A}(\alpha)$ where $\alpha$ is a generator with
  $\Gamma=0$ and $I(\alpha)= 2k$.
\end{description}
This is because by \eqref{eqn:at3}, the above minimum of
$\mc{A}(\alpha)$ is (up to $O(\epsilon)$ error) the minimum of
$\mc{\ell}(\Lambda_\alpha)$ where $\alpha$ is a generator with
$\Gamma=0$ and $I(\alpha)= 2k$.  But it follows immediately from
\eqref{eqn:T3grading} that the latter minimum is realized by a
generator $\alpha$ in which all edges of $\Lambda_\alpha$ are labeled
`$e$' and $|P_{\Lambda_\alpha}\cap\Z^2|=k+1$.

It follows from (*) that
\[
c_{[\alpha_0]}(T^3,f\lambda) \ge \ell(\Lambda_0) -
O(\epsilon).
\]
Combining with \eqref{eqn:pi} and taking $\epsilon\to 0$ proves the
desired inequality.
\end{proof}

\begin{remark}
  In principle one could compute the full ECH spectrum of $T^3$ from
  \cite[Prop.\ 8.3]{t3}, although this is not so simple.  The latter
  proposition semi-explicitly describes a basis for the ECH consisting
  of elements $p_k,u_k,v_k$ of grading $2k$ and $s_k,t_k,w_k$ of
  grading $2k+1$ for each nonnegative integer $k$.  Here $p_k$ is the
  unique class of grading $2k$ with $U^kp_k=[\emptyset]$.  In
  particular, it follows from this description that in the notation of
  Definition~\ref{def:csigma},
\[
c_{w_k} >
 c_{u_k}=c_{v_k}=c_{s_k}=c_{t_k} > c_{p_{k-1}}.
\]
In addition it follows from the computation of the $U$ map in
\cite[Lem.\ 8.4]{t3} that $c_{p_k}> c_{p_{k-1}}$, $c_{u_k}>
c_{u_{k-1}}$, and so forth.  The beginning of the full ECH spectrum is
$c_{p_0}=0$, $c_{p_1}=c_{u_0}=c_{v_0}=c_{s_0}=c_{t_0}=2$,
$c_{p_2}=c_{u_1}=c_{v_1}=c_{s_1}=c_{t_1}=c_{w_0}=2+\sqrt{2}$,
$c_{p_3}=4$, $c_{u_2}=c_{v_2}=c_{s_2}=c_{t_2}=c_{w_1}=2+2\sqrt{2}$.
\end{remark}

\subsection{Distinguished ECH spectrum of some nonstandard 3-tori}
\label{sec:dualnorm}

We now prove Theorem~\ref{thm:dual}, computing the distinguished ECH
capacities of the examples $T_{\|\cdot\|^*}$ defined in
\S\ref{sec:dualintro}.  Note that this generalizes
Proposition~\ref{prop:standardt3}, because if
$\|\cdot\|$ is the Euclidean norm on $\R^2$, then $\lambda$ restricts
to $\partial T_{\|\cdot\|^*}$ as the standard contact form
\eqref{eqn:stdt3} on $T^3$.

\begin{proof}[Proof of Theorem~\ref{thm:dual}.]
We may assume without loss of generality that the norm $\|\cdot\|$ is
smooth.  This follows from Proposition~\ref{prop:extend}, because an
arbitrary norm can be approximated from above and below by smooth
norms, and for a given positive integer $k$ the right hand side of
\eqref{eqn:dual} depends continuously on the norm.

Since the norm $\|\cdot\|$ is smooth, $T_{\|\cdot\|^*}$ is a Liouville
domain.  We now follow the proof of Proposition~\ref{prop:standardt3}
with appropriate modifications.

To start we compute the Reeb vector field of
$\lambda=\sum_{i=1}^2p_idq_i$ on $\partial T_{\|\cdot\|^*}$.  Let $B$
denote the unit ball of the dual norm $\|\cdot\|^*$; observe that
$\partial B$ is a smooth convex curve in $(\R^2)^*$.  Identify
$(\R^2)^*=\R^2$ using the usual coordinates $p_1,p_2$.  Suppose
$(p_1,p_2)\in \partial B$.  There is a unique $\theta\in\R/2\pi\Z$
such that the outward unit normal vector to $\partial B$ at
$(p_1,p_2)$ (with respect to the Euclidean metric) is given by
$(\cos\theta,\sin\theta)$.  The Reeb vector field at
$(q_1,q_2,p_1,p_2)$ is then
\[
R = \left(p_1\cos\theta +
  p_2\sin\theta\right)^{-1}\left(\cos\theta\frac{\partial}{\partial q_1} +
  \sin\theta\frac{\partial}{\partial q_2}\right).
\]
It follows that for every pair of relatively prime integers $(m,n)$
there is a Morse-Bott circle of embedded Reeb orbits $\mc{O}_{m,n}$,
sweeping out $T^2\times\{(p_1,p_2)\}$ where $(p_1,p_2)$ corresponds as
above to the unique $\theta$ satisfying \eqref{eqn:theta}.  There are
no other embedded Reeb orbits.  Each Reeb orbit
$\gamma\in\mc{O}_{m,n}$ has symplectic action
\[
\mc{A}(\gamma) = p_1m + p_2n.
\]
Now observe that since $\|\cdot\|$ is the dual norm of $\|\cdot\|^*$,
we have
\[
\left\|(m,n)\right\| = \max\left\{\langle\zeta,(m,n)\rangle\;\big|\; \zeta\in
  B\right\}.
\]
By the definition of $\theta$, this maximum is realized by
$\zeta=(p_1,p_2)$.  In conclusion, each Reeb orbit
$\gamma\in\mc{O}(m,n)$ has symplectic action
\begin{equation}
\label{eqn:dualaction}
\mc{A}(\gamma) = \left\|(m,n)\right\|.
\end{equation}

The rest of the proof is now the same as the proof of
Proposition~\ref{prop:standardt3}, with equation \eqref{eqn:aomn}
replaced by \eqref{eqn:dualaction}, and $\ell$ replaced by
$\ell_{\|\cdot\|}$.
\end{proof}

\section{The polydisk}
\label{sec:pd}

\subsection{The ECH capacities of a polydisk}

We now prove Theorem~\ref{thm:polydisk} on the (distinguished) ECH
capacities of a polydisk.  One can calculate the ECH capacities of a
polydisk by understanding the ECH chain complex of an appropriately
smoothed polydisk, similarly to the calculations in \cite{t3} for
$T^3$ as outlined in \S\ref{sec:dest3}.  However this is a long story,
and we will instead take a shortcut using Theorems~\ref{thm:main} and
\ref{thm:dual}.

\begin{proof}[Proof of Theorem~\ref{thm:polydisk}.]
The proof has two steps.

{\em Step 1.\/}
Define a norm $\|\cdot\|$ on $\R^2$ by
\begin{equation}
\label{eqn:manhattan}
\|(q_1,q_2)\| = \frac{a|q_1|}{2}+\frac{b|q_2|}{2}.
\end{equation}
The dual norm is then
\[
\|(p_1,p_2)\|^* = \max\left(\frac{2|p_1|}{a},\frac{2|p_2|}{b}\right),
\]
so that
\[
T_{\|\cdot\|^*} = \left\{(q_1,q_2,p_1,p_2)\in T^*T^2 \;\big|\; |p_1|
  \le a/2, \; |p_2| \le b/2\right\}.
\]
Denote this by $T(a,b)$.

Observe now that for any $\epsilon>0$, there is a symplectic embedding
$P(a,b)\to T(a+\epsilon,b+\epsilon)$ defined by
\[
(z_1,z_2) \longmapsto (\phi_1(z_1),\phi_2(z_2)),
\]
where $\phi_1=(p_1,q_1)$ is an area-preserving embedding of the disc of area
$a$ into the cylinder $[-(a+\epsilon)/2,(a+\epsilon)/2]\times\R/\Z$,
and $\phi_2=(p_2,q_2)$ is an area-preserving embedding of the disc of area
$b$ into the cylinder $[-(b+\epsilon)/2,(b+\epsilon)/2]\times\R/\Z$.
There is also a symplectic embedding $T(a-\epsilon,b-\epsilon)\to
P(a,b)$ defined by
\[
(q_1,q_2,p_1,p_2) \longmapsto
\pi^{-1/2}\left((a/2+p_1)^{1/2}e^{2\pi i q_1},
  (b/2+p_2)^{1/2}e^{2\pi i q_2}\right).
\]

Consequently, for any given $k$, applying Theorem~\ref{thm:main} and
taking $\epsilon\to 0$ shows that
\[
c_k(P(a,b)) = c_k(T(a,b)).
\]
So by Theorem~\ref{thm:dual}, we need to show that
\begin{equation}
\label{eqn:minmin}
\min\left\{am+bn\;\big|\;
  (m+1)(n+1)\ge k+1\right\}
=
\min\left\{\ell_{\|\cdot\|}(\Lambda) \,\big|\, |P_\Lambda\cap\Z^2|=k+1
\right\},
\end{equation}
where in the first minimum $(m,n)\in\N^2$, and in the second minimum
$\Lambda$ is a convex polygon in $\R^2$ with vertices in $\Z^2$.

{\em Step 2.\/} We now prove \eqref{eqn:minmin}.
Given a convex polygon $\Lambda$ in $\R^2$ with vertices in $\Z^2$,
let $m$ denote the horizontal displacement between the rightmost and
leftmost vertices, and let $n$ denote the vertical displacement
between the top and bottom vertices. Then $\Lambda$ is contained in a
rectangle of side lengths $m$ and $n$, so
\[
|P_\Lambda\cap\Z^2|\le (m+1)(n+1).
\]
On the other hand it follows from \eqref{eqn:manhattan} that
\[
\ell_{\|\cdot\|}(\Lambda) = am+bn.
\]
Hence the left hand side of \eqref{eqn:minmin} is less than or equal
to the right hand side.  But the reverse inequality also holds,
because if $k+1\le(m+1)(n+1)$, then inside a rectangle of side lengths
$m$ and $n$ one can find a convex polygon $\Lambda$ with
$|P_\Lambda\cap\Z^2|=k+1$.
\end{proof}

\subsection{Obstructions to embedding polydisks into balls}
\label{sec:simplecalculations}

Let us now try to more explicitly understand the bound \eqref{eqn:po}
(which we have now justified) for the function $g$ defined in
\S\ref{sec:pb} that measures the obstruction to symplectically
embedding a polydisk into a ball.
The bound \eqref{eqn:po} can be written as $g(a)\ge \sup_{d=1,2,\ldots}
g_d(a)$, where
\[
g_d(a) \eqdef \min\left\{\frac{am+n}{d} \;\bigg|\; (m,n)\in\N^2,\;\;
  (m+1)(n+1)\ge \frac{(d+1)(d+2)}{2}\right\}.
\]
Given $d$, one can compute the function $g_d$ as follows.  Let
$\Lambda_d$ denote the boundary of the convex hull of the set of
lattice points $(m,n)\in\N^2$ with $(m+1)(n+1)\ge(d+1)(d+2)/2$.  Then
$g_d(a)=(am+n)/d$, where $(m,n)$ is a (usually unique) vertex of the
polygonal path $\Lambda_d$ incident to edges of slope less than or
equal to $-a$ and slope greater than or equal to $-a$.  Using this
observation, we can now give the:

\begin{proof}[Proof of Proposition~\ref{prop:pb}.]
First consider $d=1$.  The path $\Lambda_1$ has vertices $(0,2)$,
$(1,1)$, and $(2,0)$.  Since the vertex $(0,2)$ is incident to edges
of slope $-1$ and $-\infty$, the above discussion shows that
\[
g_1(a) = 2, \quad\quad a\ge 1.
\]
This proves the first line of \eqref{eqn:ga}.  To prove the rest of
\eqref{eqn:ga}, take $d=6$.  The path $\Lambda_6$ has vertices
$(0,27)$, $(1,13)$, $(2,9)$, $(3,6)$, $(4,5)$, $(5,4)$, $(6,3)$,
$(9,2)$, $(13,1)$, and $(27,0)$.  Since the vertex $(3,6)$ is incident to
edges of slope $-1$ and $-3$, we get
\[
g_6(a) = \frac{3a+6}{6}, \quad 1\le a\le 3.
\]
This implies the second line of \eqref{eqn:ga}.  And since the vertex
$(2,9)$ is incident to edges of slope $-3$ and $-4$, we obtain
\[
g_6(a) = \frac{2a+9}{6}, \quad 3\le a\le 4.
\]
This gives the last line of \eqref{eqn:ga}.
\end{proof}

\section{Volume and quantitative ECH}
\label{sec:volume}

We now discuss and present evidence for
Conjecture~\ref{conj:Liouville} and some variants, relating the
asymptotics of quantitative ECH to volume.

\subsection{Volume conjecture for the distinguished ECH spectrum}

If $(Y,\lambda)$ is a closed contact 3-manifold, define
\[
\op{vol}(Y,\lambda) \eqdef \int_Y\lambda\wedge d\lambda.
\]
Conjecture~\ref{conj:Liouville} is then a special case of the following:

\begin{conjecture}
\label{conj:distinguishedvolume}
Let $(Y,\lambda)$ be a closed contact 3-manifold with nonvanishing ECH
contact invariant.  Suppose that $c_k(Y,\lambda)<\infty$ for all $k$.
Then
\[
\lim_{k\to\infty}\frac{c_k(Y,\lambda)^2}{k} = 2
  \op{vol}(Y,\lambda).
\]
\end{conjecture}

By Remark~\ref{rem:ellipsoidVolume} and Proposition~\ref{prop:dss3},
this conjecture holds for ellipsoids.  Here are some more examples:

\begin{example}
\label{ex:t3volume}
  Consider $T^3$ as in \eqref{eqn:t3} with the standard contact form
  $\lambda$ in \eqref{eqn:stdt3}.  Let $\Lambda$ be a convex polygon
  as in \eqref{eqn:T3spectrum}.  If $A(\Lambda)$ denotes the area
  enclosed by $\Lambda$, then
\[
|P_\Lambda\cap\Z^2| = A(\Lambda) + O(\ell(\Lambda)).
\]
It then follows from \eqref{eqn:T3spectrum} and the isoperimetric
inequality
\[
\ell(\Lambda)^2 \ge 4\pi A(\Lambda)
\]
that
\[
\liminf_{k\to\infty}\frac{c_k(T^3,\lambda)^2}{k} \ge 4\pi.
\]
On the other hand, approximating a circle with polygons shows that if
$k$ is large, then we can find a polygon $\Lambda$ as in
\eqref{eqn:T3spectrum} with
\[
\ell(\Lambda)^2 \le 4\pi A(\Lambda) + O(\ell(\Lambda)),
\]
so in fact
\[
\lim_{k\to\infty}\frac{c_k(T^3,\lambda)^2}{k} = 4\pi.
\]
Since $\op{vol}(T^3)=2\pi$, Conjecture~\ref{conj:distinguishedvolume} is
confirmed in this case.
\end{example}

\begin{example}
  More generally, let $\|\cdot\|$ be a smooth norm on $\R^2$, let $B$
  denote the unit ball in the dual norm $\|\cdot\|^*$, and consider
  the Liouville domain $T_{\|\cdot\|^*}$ from \S\ref{sec:dualnorm}.
  We have $\op{vol}\left(T_{\|\cdot\|^*}\right)=A(B)$, where $A(B)$
  denotes the area of $B$ (with respect to the Euclidean metric).  So
  it follows from Theorem~\ref{thm:dual} that
  Conjecture~\ref{conj:Liouville} in this case is equivalent to a
  sharp isoperimetric inequality
\begin{equation}
\label{eqn:wulff}
\ell_{\|\cdot\|}(\Lambda)^2\ge 4A(B)A(\Lambda)
\end{equation}
for a smooth convex curve $\Lambda$.  Now \eqref{eqn:wulff} holds
because if $A(\Lambda)$ is fixed, then $\ell_{\|\cdot\|}(\Lambda)$ is
minimized when $\Lambda$ is a scaling of a $90^\circ$ rotation of
$\partial B$, see \cite{bm,wulff}; and one can check directly that in
this case equality holds in \eqref{eqn:wulff}.
\end{example}

\begin{proposition}
\label{prop:somewhere}
If Conjecture~\ref{conj:distinguishedvolume} holds for closed contact
three-manifolds $(Y_i,\lambda_i)$ with nonvanishing contact invariant
for $i=1,\ldots,n$, then it also holds for $(Y,\lambda)\eqdef
\coprod_{i=1}^n(Y_i,\lambda_i)$.
\end{proposition}

\begin{proof}
  By Proposition~\ref{prop:dsdu}, we can assume that
  $c_k(Y_i,\lambda_i)<\infty$ for all $i$ and $k$, and we have
\[
\lim_{k\to\infty}\frac{c_k(Y,\lambda)}{\sqrt{2k}} =
\lim_{k\to\infty}\frac{1}{\sqrt{2k}}
\max_{k_1+\cdots+k_n=k}\sum_{i=1}^n\sqrt{2k_i\op{vol}(Y_i,\lambda_i)},
\]
provided that the limit on the right exists.  If one drops the
integrality requirement on $k_i$, then the maximum on the right is
attained when
\[
k_i = \frac{k \op{vol}(Y_i,\lambda_i)}{\op{vol}(Y,\lambda)}.
\]
We then obtain
\[
\lim_{k\to\infty}\frac{c_k(Y,\lambda)}{\sqrt{2k}} =
\lim_{k\to\infty}
\sum_{i=1}^n(\op{vol}(Y,\lambda))^{-1/2}\op{vol}(Y_i,\lambda_i)
= \sqrt{\op{vol}(Y,\lambda)}
\]
as required.
\end{proof}

There is also (limited) experimental support for a related conjecture:

\begin{conjecture}
\label{conj:qw}
If $(Y,\lambda)$ satisfies the assumptions of
  Conjecture~\ref{conj:distinguishedvolume}, then $ c_k(Y,\lambda) <
  \sqrt{2k\op{vol}(Y,\lambda)} $ for all $k>0$.
\end{conjecture}

\begin{remark}
\label{rem:qw}
Conjecture~\ref{conj:qw} implies quantitative refinements of the
three-dimensional Weinstein conjecture, since by definition, if
$\lambda$ is nondegenerate, then $(Y,\lambda)$ has at least $k$
nonempty ECH generators of action at most $c_k(Y,\lambda)$.  For
example, the $k=1$ case of Conjecture~\ref{conj:qw} implies that if
$(Y,\lambda)$ satisfies the hypotheses of
Conjecture~\ref{conj:distinguishedvolume}, then $\lambda$ has a Reeb
orbit of symplectic action at most $\sqrt{2\op{vol}(Y,\lambda)}$.
\end{remark}

\subsection{Volume conjecture for Liouville domains}

We now confirm Conjecture~\ref{conj:Liouville} in some more cases.

\begin{proposition}
Let $(X_0,\omega_0)$ be a 4-dimensional Liouville domain.  Then:
\begin{description}
\item{(a)}
\begin{equation}
\label{eqn:liminf}
\liminf_{k\to \infty}\frac{c_k(X_0,\omega_0)^2}{k} \ge
4\op{vol}(X_0,\omega_0).
\end{equation}
\item{(b)} Suppose that $(X_0,\omega_0)$ can be symplectically
  embedded into a 4-dimensional Liouville domain $(X_1,\omega_1)$ such that
  $c_k(X_1,\omega_1)<\infty$ for all $k$ and
  Conjecture~\ref{conj:Liouville} holds for $(X_1,\omega_1)$.  Then
  Conjecture~\ref{conj:Liouville} holds for $(X_0,\omega_0)$.
\end{description}
\end{proposition}

\begin{proof}
  (a) For any $\epsilon>0$, by using a finite cover of $X_0$ by
  Darboux charts, we can fill all but $\epsilon$ of the volume of
  $(X_0,\omega_0)$ with products of smoothed squares which are
  symplectomorphic to polydisks.  Since
  Conjecture~\ref{conj:Liouville} is true for a polydisk, by
  Proposition~\ref{prop:somewhere} (applied to boundaries of smoothed
  polydisks) it is also true for a disjoint union of polydisks.
  Applying Theorem~\ref{thm:main} then gives
\[
\liminf_{k\to \infty}\frac{c_k(X_0,\omega_0)^2}{k} \ge
4\left(\op{vol}(X_0,\omega_0) - \epsilon\right).
\]
Since $\epsilon>0$ was abitrary, this proves \eqref{eqn:liminf}.

(b) 
Fill all but volume $\epsilon$ of the complement of $X_0$ in $X_1$
by polydisks and apply Theorem~\ref{thm:main} again.
\end{proof}

\subsection{A more general volume conjecture}

Conjecture~\ref{conj:distinguishedvolume} is a special case of the
following more general conjecture.  Let $(Y,\lambda)$ be a closed
contact 3-manifold.  Recall that if $\Gamma\in H_1(Y)$ is such that
$c_1(\xi)+2\op{PD}(\Gamma)\in H^2(Y;\Z)$ is torsion, then
$ECH(Y,\lambda,\Gamma)$ has a relative $\Z$-grading, which can be
arbitrarily normalized to an absolute $\Z$-grading.  We then denote
the grading of a generator $x$ by $I(x)\in\Z$.  Recall the notation
$c_\sigma$ from Definition~\ref{def:csigma}.

\begin{conjecture}
\label{conj:general}
Let $(Y,\lambda)$ be a closed connected contact 3-manifold, let
$\Gamma\in H_1(Y)$, suppose that $c_1(\xi)+2\op{PD}(\Gamma)\in
H^2(Y;\Z)$ is torsion, and choose an absolute $\Z$-grading as above on
$ECH(Y,\lambda,\Gamma)$.  Let $\{\sigma_k\}_{k=1,2,\ldots}$ be a
sequence of elements of $ECH(Y,\lambda,\Gamma)$ with definite gradings
satisfying $\lim_{k\to\infty}I(\sigma_k)=\infty$.  Then
\begin{equation}
\label{eqn:general}
\lim_{k\to\infty}\frac{c_{\sigma_k}(Y,\lambda)^2}{I(\sigma_k)}=
\op{vol}(Y,\lambda).
\end{equation}
\end{conjecture}

Note that the validity of \eqref{eqn:general} does not depend on the
choice of absolute $\Z$-grading.  Cliff Taubes has suggested to me
that it may be possible to prove Conjecture~\ref{conj:general} using
the spectral flow estimates involved in the proof of
\eqref{eqn:echswf}.

\end{document}